\pdfoutput=1
\RequirePackage{ifpdf}
\ifpdf 
\documentclass[pdftex]{sigma}
\else
\documentclass{sigma}
\fi

\usepackage{mathrsfs}
\usepackage{mathtools}

\numberwithin{equation}{section}

\newtheorem{Theorem}{Theorem}[section]
\newtheorem*{Theorem*}{Theorem}

\newtheorem{Lemma}[Theorem]{Lemma}
\newtheorem{Proposition}[Theorem]{Proposition}
 { \theoremstyle{definition}

\newtheorem{Remark}[Theorem]{Remark} }

\def\R{\mathbb{R}}

\def\SL{\mathrm{SL}}

\def\sl{\mathfrak{sl}}

\def\d{\,\mathrm{d}}

\DeclareMathOperator{\supp}{supp}

\DeclareMathOperator{\sgn}{sgn}

\DeclareMathOperator{\re}{Re}

\DeclareMathOperator{\im}{Im}

\def\a{\alpha}
\def\b{\beta}
\def\c{\gamma}
\def\d{\delta}
\def\e{\varepsilon}
\def\f{\varphi}
\def\g{\psi}

\def\l{\lambda}
\def\m{\mu}
\def\n{\nu}

\def\s{\sigma}

\def\x{\xi}
\def\y{\eta}
\def\z{\zeta}

\def\pa{\partial}

\newcommand{\cR}{\mathcal{R}}

\newcommand{\sI}{\mathscr{I}}

\newcommand{\bfe}{\mathbf{e}}
\newcommand{\bfh}{\mathbf{h}}

\begin{document}

\newcommand{\arXivNumber}{2308.16815}

\renewcommand{\PaperNumber}{014}

\FirstPageHeading

\ShortArticleName{Strichartz Estimates for the $(k,a)$-Generalized Laguerre Operators}

\ArticleName{Strichartz Estimates for the $\boldsymbol{(k,a)}$-Generalized\\ Laguerre Operators}

\Author{Kouichi TAIRA~$^{\rm a}$ and Hiroyoshi TAMORI~$^{\rm b}$}

\AuthorNameForHeading{K.~Taira and H.~Tamori}

\Address{$^{\rm a)}$~Faculty of Mathematics, Kyushu University, 744, Motooka, Nishi-ku, Fukuoka, Japan}
\EmailD{\href{mailto:taira.kouichi.800@m.kyushu-u.ac.jp}{taira.kouichi.800@m.kyushu-u.ac.jp}}
\URLaddressD{\url{https://sites.google.com/view/the-home-page-of-kouichi-taira/home}}

\Address{$^{\rm b)}$~Department of Mathematical Sciences, Shibaura Institute of Technology,\\
\hphantom{$^{\rm b)}$}~307 Fukasaku, Minuma-ku, Saitama, 337-8570, Japan}
\EmailD{\href{mailto:tamori@shibaura-it.ac.jp}{tamori@shibaura-it.ac.jp}}

\ArticleDates{Received June 24, 2024, in final form February 12, 2025; Published online March 02, 2025}

\Abstract{In this paper, we prove Strichartz estimates for the $(k,a)$-generalized Laguerre operators $a^{-1}\bigl(-|x|^{2-a}\Delta_k+|x|^a\bigr)$ which were introduced by Ben Sa\"{\i}d--Kobayashi--{\O}rsted, and for the operators $|x|^{2-a}\Delta_k$. Here $k$ denotes a non-negative multiplicity function for the Dunkl Laplacian $\Delta_k$ and $a$ denotes a positive real number satisfying certain conditions. The cases $a=1,2$ were studied previously. We consider more general cases here. The proof depends on symbol-type estimates of special functions and a discrete analog of the stationary phase theorem inspired by the work of Ionescu--Jerison.}


\Classification{35Q41; 22E45}

\section{Introduction}

For the usual Laplacian $\Delta$ on $\R^n$ ($n\ge 1$), the following inequalities hold:
\begin{align*}
\big\|{\rm e}^{{\rm i}t\Delta}u_0\big\|_{L^p(\R;L^q(\R^n))}\leq C\|u_0\|_{L^2(\R^n)},
\end{align*}
where $(p,q)\in [2,\infty]^2$ satisfies $2/p+n/q=n/2$ with $(p,q,n)\neq (2,\infty,2)$. These estimates are called Strichartz estimates and have been widely studied in the past thirty years. Strichartz~\cite{St} proved them for $p=q$ by using the Fourier restriction estimates and a duality argument. The most difficult part, that is the end-point case $(p,q)=(2,2n/(n-2))$ with $n\geq 3$, was proved by Keel and Tao~\cite{KT}.
These are used for well-posedness of linear and non-linear time-dependent Schr\"odinger equations~\cite{GV,Y}. See also the book~\cite{T}. For the Harmonic oscillator~\smash{$H_{\rm os}=\frac{-\Delta+|x|^2}{2}$}, a similar estimates hold:
\[
\big\|{\rm e}^{-{\rm i}tH_{\rm os}}u_0\big\|_{L^p([-T,T];L^q(\R^n))}\leq C_T\|u_0\|_{L^2(\R^n)},
\]
 where the region $[-T,T]$ cannot be replaced by $\R$ essentially due to the existence of $L^2$-eigenfunctions.

Given a root system $\cR$ in $\R^n$ (we assume reducedness and do not assume crystallographic condition for the definition of root system, see \cite[Definition~2.1]{BKO}), a $[0,\infty)$-valued function on $\cR$ which is invariant under the finite reflection group $\mathfrak{C}$ associated with $\cR$ is called a non-negative multiplicity function.
For a non-negative multiplicity function $k$ and $a>0$, we define the~$(k,a)$-generalized Laguerre operator by
\begin{align*}
H_{k,a}:=\frac{-|x|^{2-a}\Delta_k+|x|^a}{a}\qquad \text{on} \ \R^n.
\end{align*}
Here $\Delta_k$ denotes the Dunkl Laplacian (see \cite[formula~(2.9)]{BKO}).
If $k\equiv 0$, the Dunkl Laplacian coincides with the usual Laplacian: $\Delta_0=\Delta=\sum_{j=1}^n\pa_{x_j}^2$.

The $(k,a)$-generalized Laguerre operator $H_{k,a}$ is the generator of the $(k,a)$-generalized Laguerre semigroup which is a holomorphic semigroup introduced by Ben Sa\"{\i}d, Kobayashi and {\O}rsted~\cite{BKO}.
For the case $k\equiv 0$ and $a=2$ (resp.\ $a=1$), the semigroup is the Hermite semigroup~\cite{Fol, H} (resp.\ the Laguerre semigroup \cite{KM1, KM, KM2}).
These two semigroups are associated with some realization (called the Schr\"odinger model) of minimal representations of the metaplectic group $\mathrm{Mp}(n,\R)$ and a double cover of the indefinite orthogonal group $\mathrm{O}(n+1,2)$.
As unitary representations of $\widetilde{\SL}(2,\R)\times\mathfrak{C}$, they deformed these two representations with parameters $k$ and~$a$, and obtained a family of unitary representations.
The $(k,a)$-generalized Laguerre semigroups are associated with them.

In \cite{BKO}, the Fourier transforms associated with the $(k,a)$-generalized Laguerre semigroups are introduced and these various properties are studied. Recently, there have been several studies related to these operators such as real Paley--Wiener theorem \cite{LLF}, $L^p$-$L^q$-boundedness of Fourier multipliers \cite{KR}, Hardy inequality \cite{Te} and wavelet transform~\cite{BMMS}.

The aim of this paper is to prove Strichartz estimates of Schr\"odinger equations associated with the $(k,a)$-generalized Laguerre operators. This problem is proposed in \cite{BKO} and solved in~\cite{B} and~\cite{M} for $a=1$ and $a=2$ (see also \cite{MS}, where they deal with orthonormal Strichartz estimates). The $(0,2)$-generalized Laguerre operator $H_{0,2}$ is just the Harmonic oscillator $H_{\rm os}$ and so their results are a generalization of the classical result for $H_{\rm os}$. Here, we deal with more general cases.

For $a=1,2$, the integral kernel of the Schr\"odinger propagator has a nice expression (due to \cite[formula~(4.58)]{BKO}), which immediately implies the dispersive estimate \cite[Proposition 4.26]{BKO}. Hence the Strichartz estimates for $a=1,2$ are a direct consequence of this estimate and the result in~\cite{KT} by Keel and Tao, see also \cite{B,MS}. One of the difficulties to extend it to general $a$ is a lack of such a nice expression of the Schr\"odinger propagator. Actually, this is just expressed in terms of an infinite sum of a product of special functions (see \eqref{sIdefn} and \eqref{intkergen}). Therefore, we need to control this sum uniformly with respect to some parameters.
To overcome this difficulty, we use a strategy inspired by the proof of Carleman estimates due to Ionescu and Jerison \cite{IJ}. To be precise, we reduce estimates of the sum to those of integrals and use the theory of oscillatory integrals such as the stationary phase theorem with several parameters. One difference from \cite{IJ} is that we avoid using the dyadic decomposition which is used there many times. Instead, we employ an appropriate scaling and simplify some arguments.
In a sequel work \cite{TT}, we will give another approach based on a deformation of integrals developed in the proof for the Strichartz estimates on flat cones \cite{F}.

In the last few decades, numerous works have focused on the dispersive estimates or the Strichartz estimates for Schr\"odinger operators with critical electromagnetic potentials such as Aharonov--Bohm magnetic fields (see \cite{FFFP1,FFFP2,FZZ,GYZZ}) and Laplace--Beltrami operators on conic manifolds (\cite{F,ZZ}). Due to the spherical symmetry of their operator, the integral kernel of their propagator has the form
\begin{align*}
\sum_{\n\colon \text{eigenvalues on the sphere}}K_{\n}(t,r_1,r_2)H_{\n}(\theta_1,\theta_2),
\end{align*}
where $K_{\n}$ is the propagator in the radial direction and $H_{\n}$ is the projection in the spherical direction. To achieve optimal estimates for this integral kernel, one has to use the oscillatory behavior of $K_{\n}$ and $H_{\n}$, in other words, some cancelation of the sum much like our case. In their paper, it is accomplished through the use of the complex contour deformation and certain functional equations of the special functions.
Therefore, their method appears to be inapplicable when the radial propagator or the spherical projection is not expressed by special functions. Typical scenarios involving these situations arise in the studies of the Schr\"odinger equation with the degenerate trapping \cite{C} or the wave equation on the Schwartzshild spacetime \cite{DSS}. In their works, the dispersive or Strichartz estimates for initial values with a fixed angular momentum are considered, that is, they studied single modes only and did not consider the sum possibly because of the difficulty to treat the oscillation of the sum.

In this paper, we deduce the asymptotic expansions of the radial direction and the spherical projection first and then sum up them by exploiting their oscillatory behavior. Specifically, we use the properties of the special functions in the first step only. Hence the authors believe that the method employed here specifically in the second step remains applicable even when $K_{\n}$ and~$H_{\n}$ are not expressed in terms of special functions as is the case in \cite{C,DSS} (although we might need more precise analysis of the radial propagators).

The second contribution of this paper is to give a naive application of the stationary phase theorem (see Proposition \ref{Stphasemovecrit}). This is used to prove improved dispersive estimates for $H_{k,a}$ (near the diagonal) under the restriction $0<a<2$. Our integral kernel has multiple parameters and it seems important to consider when a similar statement the usual stationary phase theorem holds uniformly with respect to additional parameters and when it is improved. Our Proposition~\ref{Stphasemovecrit} addresses intermediate cases between two scenarios where the decay order of an oscillatory integral is improved, see Section \ref{subsecStphase}.

Finally, we also obtain symbol-type estimates of higher-order derivatives for $J$-Bessel functions, which was done in \cite{IJ} up to second derivatives and was anticipated to be true for higher-order derivatives there (see \cite[Remark after Theorem 9.1]{IJ}). It seems that the method used in \cite{IJ} via complex counter deformation cannot be applied to the estimates for higher-order derivatives. Here we use an alternative method based on the stationary phase type theorem and partially solve them at the cost of loss of estimates for some parameters.
See Proposition \ref{Besselasymp} for the precise statement and Appendix~\ref{AppBessel} for its proof. We also mention a recent work~\cite{Sh}, where the precise asymptotic behavior of the Bessel function is given although the authors do not know whether our symbol-type estimates follow from the results in~\cite{Sh}.

\subsection{Main theorem}
Let us state our main theorem.
For a nonnegative multiplicity function $k$ and $a>0$, we write
\begin{align}\label{vartheta}
\vartheta_{k,a}(x):=|x|^{a-2}\prod_{\alpha\in\cR}|\langle\alpha,x\rangle|^{k(\alpha)}
\end{align}
and assume that the homogeneous degree of the measure $\vartheta_{k,a}(x){\rm d}x$ on $\R^n$ is positive
\begin{align}\label{eq:positive}
\sigma_{k,a}:=\frac{n+\sum_{\alpha\in\cR}k(\alpha)+a-2}{a}>0.
\end{align}
Then it is shown in \cite[Corollary 3.22]{BKO} that the operator $H_{k,a}$ with domain $W_{k,a}(\R^n)$ defined in~\cite[equation~(3.29)]{BKO} is essentially self-adjoint on the Hilbert space $L^2(\R^n; \vartheta_{k,a}(x) {\rm d}x)$. We denote the unique self-adjoint extension of $H_{k,a}$ by the same symbol $H_{k,a}$.

We write $L^q=L^q(\R^n;\vartheta_{k,a}(x){\rm d}x )$ and $L^p(I, L^q)=L^p(I; L^q(\R^n;\vartheta_{k,a}(x){\rm d}x ))$ for $I\subset \R$.
Recall that an exponent pair $(p,q)\in[2,\infty]^2$ is called \emph{$\sigma_{k,a}$-admissible} if
\begin{align*}
\frac{1}{p}+\frac{\s_{k,a}}{q}=\frac{\s_{k,a}}{2} \qquad \text{and}\qquad (p,q,\s_{k,a})\neq (2,\infty,1).
\end{align*}

When $n=1$, a nonnegative multiplicity function $k$ is a constant function.
In this case, we regard $k$ as a nonnegative real number $k(\alpha)$ $(\alpha\in\cR)$, and we see $\sum_{\alpha\in\cR}k(\alpha)=2k$.

\begin{Theorem}\label{mainthmSt}
We assume one of the following, which implies \eqref{eq:positive}:
\begin{itemize}\itemsep=0pt
\item $n=1$ and $a\ge 2-4k$,
\item $n\ge 2$ and $(0<a\le 1$ or $a=2)$,
\item $n\ge 2$, $1<a<2$ and $k\equiv 0$.
\end{itemize}
Let $(p,q), (p_1,q_1), (p_2,q_2)\in [2,\infty]^2$ be $\sigma_{k,a}$-admissible exponents.
Then, for $T>0$, there exists~${C>0}$ such that
\begin{align}
&\big\|{\rm e}^{-{\rm i}tH_{k,a}}u\big\|_{L^p([-T,T]; L^q)}\leq C\|u\|_{L^2}\label{homStr},\\
&\left\|\int_{0}^t{\rm e}^{-{\rm i}(t-s)H_{k,a}}f(s)ds \right\|_{L^{p_1}([-T,T]; L^{q_1})} \leq C\|f\|_{L^{p_2^*}([-T,T]; L^{q_2^*})},\label{inhomStr}
\end{align}
where $r^*$ denotes the H\"older conjugate of $r$: $r^*=r/(r-1)$ and in addition $(p_j,q_j)\neq \bigl(2,\frac{2\s_{k,a}}{\s_{k,a}-1}\bigr)$ when all of the conditions $n\geq 2$, $1<a<2$ and $k\equiv 0$ hold.
\end{Theorem}

\begin{Remark}\label{rem:mainthmrem}\quad
\begin{itemize}\itemsep=0pt
\item[(1)]
The $a=1$ case was treated in \cite{B}, where an additional assumption $\s_{k,a}\ge 1$ is necessary in order to use an upper estimate of \smash{$\sI\bigl(2,\frac{\s_{k,a}-1}{2};w;t\bigr)$} \cite[Proposition 4.26]{BKO} (see \eqref{sIdefn} for the definition of $\sI$) although it is not explicitly written there. When $n=1$, the assumption~${\s_{k,a}\ge 1}$ implies our assumption $a\ge 2-4k$. Moreover, the end-point case \smash{$\bigl(2,\frac{2\s_{k,a}}{\s_{k,a}-1}\bigr)$} is excluded in \cite{B}.
The end-point case follows from the result in \cite{KT}.

\item[(2)] For $n\geq 2$ with $a\leq 1$, we also obtain a dispersive estimate, see the proof of Theorem~\ref{mainthmSt} in Section \ref{subsec:Stpf}. On the other hand, the dispersive estimate might break for $n\geq 2$ with~${1<a<2}$ (see Theorem \ref{keythm}\,(i) and Remark \ref{dispremark}). Nevertheless, the Strichartz estimates still hold if $k\equiv 0$ since its proof just relies on the dispersive estimate around the diagonal of the integral kernel due to the nature of the $TT^*$ argument. The authors believe that the Strichartz estimates do not hold for $a>2$ although they do not know its proof.

\item[(3)] We exclude an inhomogeneous end-point estimate for $1<a<2$ with $n\geq 2$ and $k\equiv 0$ since a global dispersive estimate is absent. A technique used in \cite{HZ,ZZ} might be available, however, our estimates are not sufficient to apply their method.

\item[(4)] In the above estimates, we cannot replace the time interval $[-T,T]$ by $\R$. In fact, we take~${u\neq 0}$ be an $L^2$-eigenfunction of $H_{k,a}$ and we denote the corresponding eigenvalue by~${\l\in \R}$ (note that its spectrum is discrete, see \cite[Corollary 3.22]{BKO}). Then \eqref{homStr} implies~${u\in L^q}$ since $\big|{\rm e}^{-{\rm i}tH_{k,a}}u(x)\big|=\big|{\rm e}^{-{\rm i}t\l}u(x)\big|=|u(x)|$. On the other hand,
 \[
 \big\|{\rm e}^{-{\rm i}tH_{k,a}}u\big\|_{L^{p}(\R;L^q)}=\|u\|_{L^{p}(\R;L^q)}=\infty\] although $\|u\|_{L^2}<\infty$.
\end{itemize}
\end{Remark}

Moreover, we can deduce global in time Strichartz estimates for $-|x|^a\Delta_{k}$.

\begin{Theorem}\label{coroSt}
Under the same assumptions and notation as Theorem {\rm\ref{mainthmSt}}, there exists $C>0$ such that
\begin{align*}
&\big\|{\rm e}^{{\rm i}t|x|^a\Delta_k}u\big\|_{L^p(\R; L^q)}\leq C\|u\|_{L^2},\qquad\left\|\int_{0}^t{\rm e}^{{\rm i}(t-s)|x|^a\Delta_k}f(s){\rm d}s \right\|_{L^{p_1}(\R; L^{q_1})} \leq C\|f\|_{L^{p_2^*}(\R; L^{q_2^*})}.
\end{align*}
\end{Theorem}

\subsection{Key theorem}
We write $J_{\lambda}(z)$ for the $J$-Bessel function, $\tilde{I}_{\lambda}(w)$ for the normalized $I$-Bessel function, and $C^{\nu}_m(t)$ for
the Gegenbauer polynomial of degree $m$.
These functions are defined by
\begin{align*}
\tilde{I}_{\lambda}(w):={}&{\rm e}^{-\frac{\pi}{2}{\rm i}\l}\left(\frac{w}{2}\right)^{-\l}J_{\l}({\rm i}w),\\
C^{\nu}_m(t):={}&\frac{(-2)^m}{m!}\frac{\Gamma(m+\nu)\Gamma(m+2\nu)}{\Gamma(\nu)\Gamma(2m+2\nu)}(1-t^2)^{-\nu+\frac{1}{2}}\frac{{\rm d}^m}{{\rm d}t^m}\bigl(1-t^2\bigr)^{m+\nu-\frac{1}{2}}.
\end{align*}
Now we define
\begin{align}\label{sIdefn}
\sI(b,\nu;w;t):={}&
\Gamma(b\nu+1)\sum_{m=0}^{\infty}\left(\frac{w}{2}\right)^{bm}\tilde{I}_{b(m+\nu)}(w)(m+\nu)\n^{-1}C^{\nu}_m(t)\\
={}&\Gamma(b\nu+1)\sum_{m=0}^{\infty}\left(\frac{{\rm i}w}{2}\right)^{-b\nu}{\rm e}^{-\frac{\pi}{2}bm {\rm i}}J_{b(m+\nu)}({\rm i}w)(m+\nu)\n^{-1}C^{\nu}_m(t)\notag
\end{align}
for $b>0,\n\geq -\frac{1}{2}$, $w\in\mathbb{C}\setminus (-\infty,0)$ and $t\in [-1,1]$, where we interpret $(m+\nu)\nu^{-1}C^{\nu}_m(t)$ as~${\lim_{\n\searrow 0}(m+\nu)\n^{-1}C^{\nu}_m(t)}$ when $\n=0$. Moreover, we take a branch of $w^{bm}$ such that $w^{bm}\in (0,\infty)$ for $w\in (0,\infty)$ and $w^{bm}|_{w=0}=0$.
Then the sum in \eqref{sIdefn} absolutely converges, and $\sI$ is a~continuous function (see \cite[Lemma 4.17\,(1)]{BKO} or Section \ref{subsec:proof}).

In \cite[Theorem C, equations (4.50) and (4.52)]{BKO}, it is shown that the integral kernel ${\rm e}^{-{\rm i}tH_{k,a}}(x,\allowbreak x')$ of ${\rm e}^{-{\rm i}tH_{k,a}}$ for $0<|t|<\pi$ is given by
\begin{align}\label{intkergen}
c_{k,a}\frac{{\rm e}^{{\rm i}\frac{|x|^a+|x'|^a}{a}\mathrm{cot}(t)} }{({\rm i}\sin (t))^{\sigma_{k,a}}} \int_{\R^n}\sI\Bigg(\frac{2}{a},\frac{a(\sigma_{k,a}-1)}{2}; -{\rm i}\frac{2|x|^{\frac{a}{2}}|x'|^{\frac{a}{2}} }{a\sin(t)}; \x\cdot \hat{x'}\Bigg){\rm d}\m_{\hat{x}}^k(\x),
\end{align}
where $\hat{x}=x/|x|$, $c_{k,a}$ is the constant defined in \cite[equation~(1.6)]{BKO}, $\cdot $ denotes the standard inner product on $\R^n$ and ${\rm d}\m_{\hat{x}}^{k}$ is the probability measure introduced in \cite[equation~(2.5)]{BKO}. Moreover, for $k\equiv 0$, we have a more explicit expression
\begin{align}\label{intker0}
{\rm e}^{-{\rm i}tH_{0,a}}(x,x')=c_{0,a}\frac{{\rm e}^{{\rm i}\frac{|x|^a+|x'|^a}{a}\mathrm{cot}(t)} }{({\rm i}\sin (t))^{\frac{n+a-2}{a}}} \sI\left(\frac{2}{a},\frac{n-2}{2}; -{\rm i}\frac{2|x|^{\frac{a}{2}}|x'|^{\frac{a}{2}} }{a\sin(t)}; \hat{x}\cdot \hat{x'}\right).
\end{align}

Theorem \ref{mainthmSt} for $n\geq 2$ is a consequence of uniform bounds for $\sI(b,\nu;w;t)$:

\begin{Theorem}\label{keythm}
Let $\n\geq 0$.
\begin{enumerate}\itemsep=0pt
\item[$(i)$]
Suppose $0<b<2$ and $\e>0$. Then there exists $C_{b,\n,\e}>0$ such that
\begin{align*}
|\sI(b,\nu;-{\rm i}y;\cos\f)|\leq C_{b,\n,\e}(1+|y|)^{(1-b)\n}\qquad \text{for}\quad y\in \R,\quad \f\in [0,\pi-\e].
\end{align*}
When $b=1$, we can take $\e=0$.

\item[$(ii)$]
Suppose $b>0$. Then there exists $C_{b,\n}>0$ such that
\begin{align*}
|\sI(b,\nu;-{\rm i}y;\cos\f)|\leq C_{b,\n}(1+|y|)^{(2-b)\n}\qquad \text{for}\quad y\in \R,\quad \f\in [0,\pi].
\end{align*}
\end{enumerate}
In particular, when $1< b< 2$, the sum $\sI(b,\nu;-{\rm i}y;\cos\f)$ is uniformly bounded with respect to~${y\in \R}$ and $\f\in [0,\pi-\e]$. Moreover, when $b=1$, $b\geq 2$ or $\nu=0$, the sum $\sI(b,\nu;-{\rm i}y;\cos\f)$ is uniformly bounded with respect to $y\in \R$ and $\f\in [0,\pi]$.
\end{Theorem}

\begin{Remark}\label{dispremark}
These estimates are sharp for $b=1,2$ with respect to the growth in $y$. In fact, \cite[equations~(4.45), (4,46)]{BKO} show $|\sI(1,\nu;-{\rm i}y;\cos\f)|=1$ and
\[
\sI(2,\nu;-{\rm i}y;-1)=\Gamma\left(\n+\frac{1}{2}\right)\tilde{I}_{\n-\frac{1}{2}}(0)=1,
\]
 which do not decay in $y$. The authors believe that they are sharp also for general $b>0$. We will pursue it in our sequel work \cite{TT}.
\end{Remark}

\subsection{Idea of the proof}

Here we give an idea of the proof of Theorem \ref{keythm}. We remark that the difficulty lies in the uniformity with respect to the parameters $y$ and $\f$. For the case $|y|\lesssim 1$, the results are an immediate consequence of the estimates given in \cite{BKO} (for rigorous treatment, see Section \ref{subsec:proof}). Hence we consider the case $|y|\gtrsim 1$. Let us assume $y\gtrsim 1$ for simplicity.

First, we try to estimate the sum $\sI$ using optimal estimates for the Bessel functions and the Gegenbauer polynomials. We write
\begin{align*}
\sI(b,\nu;- {\rm i}y;\cos\f)=&L_{b,\n}y^{-b\n}\sum_{m=0}^{\infty}(m+\nu){\rm e}^{-\frac{\pi}{2}bm {\rm i}}J_{b(m+\n)}(y)\n^{-1}C^{\nu}_m(\cos\f).
\end{align*}
Since the estimates for finite $m$ are easy to prove, we only consider the sum over $m\gg 1$.
By the bounds for the Bessel functions $|J_{\m}(y)|\leq C\m^{-\frac{1}{3}}$ in \eqref{Besselunif} and the Gegenbauer polynomials~${\big|\n^{-1}C_m^{\n}(\cos \f)\big|\leq Cm^{2\n-1}}$ in \eqref{Gegenunif}, we have
\begin{align*}
\left|y^{-b\n}\sum_{m\gg 1}^{\infty}(m+\nu){\rm e}^{-\frac{\pi}{2}bm {\rm i}}J_{b(m+\n)}(y)\n^{-1}C^{\nu}_m(\cos\f)\right|\leq{}& Cy^{-b\n}\sum_{m\gg 1}m\cdot m^{-\frac{1}{3}}\cdot m^{2\n-1}\\
={}&Cy^{-b\n}\sum_{m\gg 1}m^{2\n-\frac{1}{3}}.
\end{align*}
Since the sum $\sum_{m\gg 1}m^{2\n-\frac{1}{3}}$ is not convergent for $\n\ge -1/3$ at all and since the upper bounds for special functions are optimal, such a direct method cannot be applied. Similarly, the estimate based on 
$|\tilde{I}_{\lambda}(-{\rm i}y)|\le \Gamma(\lambda+1)^{-1}$
is far from a uniform estimate.

The drawback of the above strategy is to use the triangle inequality $|\sum \cdot|\leq \sum|\cdot|$.
To obtain a better estimate of the sum $\sI(b,\nu;-{\rm i}y;\cos\f)$, we need to make use of some cancellation of the sum instead of using the triangle inequality.
To do this, we write it as a sum of a WKB form
\begin{align*}
(m+\nu){\rm e}^{-\frac{\pi}{2}bm{\rm i}}J_{b(m+\n)}(y)\n^{-1}C^{\nu}_m(\cos\f)=:\z(m,y,\f){\rm e}^{{\rm i}S(m,y,\f)},
\end{align*}
where a phase function $S$ is a real-valued function and an amplitude $\z$ does not oscillate as~${m\to \infty}$ (uniformly in other parameters $y$, $\f$). Roughly speaking, it follows from the classical formula \eqref{discon} that we can replace the discrete sum by a sum of integrals of the form
\begin{align}\label{Ideaint}
\int_{m\gg 1}\z(m,y,\f){\rm e}^{{\rm i}S(m,y,\f)+2\pi {\rm i} mq} {\rm d}m,\qquad q\in\mathbb{Z}.
\end{align}
The problem on convergence (or uniform boundedness in some parameters) of such an integral has a long history and is related to the stationary phase theorem \cite{S}. We can anticipate that a uniform estimate for the sum $\sI$ can be proved by this theorem. On the other hand, we have to be careful to justify it because of the following reasons:
\begin{itemize}\itemsep=0pt
\item To use cancelation of the oscillation via the stationary or non-stationary phase theorems, we need symbol-type (derivative) estimates of $\z$. To do this, we have to study those of the special functions $J_{\m}$ and $C_{m}^{\n}$.

\item The integral \eqref{Ideaint} has a lot of parameters and we have to estimate it in a uniform way.

\item The $J$-Bessel function $J_{b(m+\n)}(y)$ has different asymptotic behaviors as $m,y\to \infty$ on the regions $m\ll y$, $m\thickapprox y$, $m\gg y$ (see Proposition \ref{Besselasymp}) and therefore the WKB form~${\z(m,y,\f){\rm e}^{{\rm i}S(m,y,\f)}}$ has different properties on each region.

\item The decay rate of $C^{\nu}_m(\cos\f)$ as $m\to \infty$ becomes worse as $\f\to 0,\pi$ (see Proposition~\ref{Gegenasymp}) and hence estimates for $\f=0,\pi$ are also worse in general. Nevertheless, at $\f=0$ with~${0<b<2}$, we can improve the estimate since the first derivative of $S(m,y,0)$ vanishes only on regions where $m$ is small enough (see Theorem \ref{keythm}\,(i) and Section \ref{subsec0b2imp}).
\end{itemize}

\subsection{Related problems}

In this subsection,\footnote{The authors would like to thank the anonymous referees for suggesting several unsolved problems and inspiring the writing of this subsection.} we mention a few related problems, including those concerning Strichartz estimates and fractional operators.
\begin{itemize}\itemsep=0pt
\item Prove the optimality of our Strichartz estimates (see Theorem~\ref{mainthmSt}).
\item Prove the Strichartz estimates for fractional operators such as $H_{k,a}^{\a}$ or $(-\Delta_{k,a})^{\a}$ ($\a>0$). When $\a=1/2$, this corresponds to the Strichartz estimates for the wave equation.
\item Study properties of the integral kernel $E(x,y)$ of the spectral projection for $H_{k,a}$ and~$-\Delta_{k,a}$. In \cite{HZ}, the Strichartz estimates on asymptotically conic manifolds are proved by using the expression of $E(x,y)$ and the stationary phase theorem.
\end{itemize}

\subsection{Notation}
Let
\begin{align*}
\mathbb{N}=\{0,1,2,\hdots\},\qquad \mathbb{N}^*=\{1,2,3,\hdots\}, \qquad \R_{+}=(0,\infty),\qquad \R_-=(-\infty,0).
\end{align*}
For a parameter $i\in I$ and $A_i,B_i$, we write $A_i\lesssim B_i$ for $i\in I$ if there is a constant $C>0$ independent of $i\in I$ such that $A_i\leq CB_i$. We denote $A_i\sim B_i$ if both $A_i\lesssim B_i$ and $B_i\lesssim A_i$ hold for $i\in I$.

\section{Preliminary}

\subsection{A sum of monotonic functions}

The next lemma is elementary and follows from the piecewise quadrature.

\begin{Lemma}\label{monotonesum}
Let $f\colon\R\to [0,\infty)$ be a continuous monotonically decreasing function. Then
\begin{align*}
\sum_{m\in b\mathbb{N},\, m_1\leq m\leq m_2}f(m)\leq \int_{m_1-b}^{m_2}f(x){\rm d}x
\end{align*}
for $b>0$ and $m_1,m_2\in \mathbb{R}$ with $m_1<m_2$, where $b\mathbb{N}=\{b n\mid n\in \mathbb{N}\}$.
\end{Lemma}

\subsection{Discrete oscillatory integrals}

In this paper, the following classical formula plays an important role:
\begin{align}\label{discon}
\sum_{m\in\mathbb{Z}}{\rm e}^{{\rm i}S(m)}\z(m)=\int_{\R}{\rm e}^{{\rm i}S(m)}\z(m){\rm d}m-\frac{1}{2\pi {\rm i}}\sum_{q\in\mathbb{Z}\setminus \{0\}}\frac{1}{q}\int_{\R}{\rm e}^{{\rm i}(S(m)+2\pi qm)}\tilde{\z}(m) {\rm d}m,
\end{align}
where $\tilde{\z}(m)=\z'(m)+{\rm i}S'(m)\z(m)$ (for example, see \cite[Proof of Lemma 5.3]{IJ}). In this paper, we use this formula for a compactly supported smooth function $\z$ and a smooth function $S$, so the convergence of the sum does not matter.
 From this formula, we can deduce a discrete analogue of the non-stationary phase theorem.

\begin{Proposition}\label{discsum1}
Let $C_{\a}>0$ for any nonnegative integer $\a$ and $M\geq 1$, $0<r\leq 1$, $0<\rho\leq 1$, $k\in \R$.
Suppose that
$
\supp \z\subset [0,M]$, $ |\pa_m^{\a}\z(m)|\leq C_{\a}M^{k-\rho\a}$.
Suppose that a smooth real-valued function $S\in C^{\infty}(\R;\R)$ satisfies
\begin{align*}
\mathrm{dist}(\pa_mS(m), 2\pi\mathbb{Z})\geq r,\qquad \big|\pa_m^{\a+1}S(m)\big|\leq C_{\a}M^{-\rho\a}|\pa_mS(m)|.
\end{align*}
Then, for each $N>0$,
\begin{align*}
\left|\sum_{m=1}^{\infty}{\rm e}^{{\rm i}S(m)}\z(m)\right|\leq C_NM^{k+1}(1+r M^{\rho})^{-N},
\end{align*}
where $C_N>0$ depends only on $0<\rho\leq 1$, $N>0$ and a finite number of $C_{\a}$.
\end{Proposition}

\begin{Remark}
In \cite{IJ}, the fact that $C_N$ is independent of $r$ is important since they avoid to use the stationary phase theorem there. In this paper, we do not use this fact, that is, we use the case $r=1$ only in this paper.
\end{Remark}

\begin{proof}[Proof of Proposition~\ref{discsum1}]
The proof is a slight modification of \cite[Lemma 5.3]{IJ} and we discuss briefly here.

By the assumption and the intermediate value theorem, there exists a unique $k\in\mathbb{Z}$ such that $\pa_mS(m)\in (2\pi k,2\pi (k+1))$. This implies
$r\leq \pa_mS(m)-2\pi k\leq 2\pi -r$. By replacing $S(m)$ by $S(m)-2\pi k$, we may assume $k=0$ and $r\leq \pa_mS(m)\leq 2\pi -r$.

From the formula \eqref{discon}, the problem reduces to the estimates for each integrals appearing in~\eqref{discon}. By the change of variable from $m$ to $M^{\rho}m$, we have
\begin{align}\label{discsumpfsc}
\int_{\R}{\rm e}^{{\rm i}S(m)}\z(m){\rm d}m=M^{\rho}\int_{\R}{\rm e}^{{\rm i}S_M(m)}\z_M(m){\rm d}m,
\end{align}
where we set $S_M(m)\!=S(M^{\rho}m)$ and $\z_M(m)\!=\z(M^{\rho}m)$. Then the assumptions imply $|\pa_m^{\a}\z_M(m)|\allowbreak\leq C_{\a}M^k$, $\big|\pa_m^{\a+1}S_M(m)\big|\leq C_{\a}|\pa_mS_M(m)|$, $\pa_mS_M(m)\geq rM^{\rho}$ and $\supp \z_M\subset \bigl[0,4M^{1-\rho}\bigr]$. We integrate by parts $N$ times in the right-hand side of \eqref{discsumpfsc} (use the identity $({\rm i}S_M'(m))^{-1} \pa_m{\rm e}^{{\rm i}S_M(m)}={\rm e}^{{\rm i}S_M(m)} $ and Lemma \ref{Leibnizrule}) and obtain
\begin{align*}
\left|\int_{\R}{\rm e}^{{\rm i}S(m)}\z(m){\rm d}m\right|\leq M^{\rho}\cdot C_{N}M^k(1+rM^{\rho})^{-N}\cdot M^{1-\rho}=C_NM^{k+1}(1+rM^{\rho})^{-N},
\end{align*}
where the term $M^{1-\rho}$ comes from the volume of the support $\supp \z_M$. The second terms of the right-hand side of \eqref{discon} are similarly dealt with as in the proof of \cite[Lemma 5.3]{IJ}.
\end{proof}

\subsection{A variant of the stationary phase theorem}\label{subsecStphase}

Here we give a variant of the stationary phase theorem which is used for an improvement of estimates of $\sI$ described in Section \ref{subsec0b2imp}.
We consider the following integral with parameters~$\l$,~$\f$:
\begin{align*}
I(\l,\f):=\int_{\R}{\rm e}^{{\rm i}\l S(\m,\f)}\c(\m,\l,\f){\rm d}\m
\end{align*}
and its decay rate with respect to $\l\gg 1$

If we freeze the parameter $\f$, then the stationary phase theorem just implies that if $\m\mapsto S(\m,\f)$ is a Morse function (that is, all critical points are non-degenerate in the sense that $\pa_{\m}^2S\neq 0$ there), the optimal decay rate of $I$ is \smash{$\l^{-\frac{1}{2}}$}. To obtain an improved decay, here we consider the following two scenarios:
\begin{itemize}\itemsep=0pt
\item If we assume that $\c$ vanishes with order $2\n$ at all the critical points of $S$ in addition, then the optimal decay rate is improved to be \smash{$\l^{-\frac{1}{2}-\n}$}.
\item If the symbol $\c$ itself decays like $(1+\l)^{-\n}$, then the decay order becomes \smash{$\l^{-\frac{1}{2}-\n}$}.
\end{itemize}
The following proposition interpolates the above two situations in a uniform way with respect to the parameter $\f$.
A typical example of such a phase function is $S(\m,\f)=(\m-\f)^2$. This simple case could simplify the reading of the proof below.

\begin{Proposition}\label{Stphasemovecrit}
Let $\n\geq 0$ and $c, C>0$. Suppose that $S\in C^{\infty}\bigl([0,1]^2;\R\bigr)$ and $\c(\cdot,\l,\cdot)\in C^{\infty}(\R\times [0,1])$ satisfy
\begin{align*}
&\supp \c(\cdot,\l,\f)\subset [0,1],\qquad |\pa_{\m}^{\a}\c(\m,\l,\f)|\leq C_{\a}\m^{2\n-\a}(1+\l \f\m)^{-\n},\\
&
\big|\pa_{\m}^2S(\m,\f)\big|\geq c,\qquad |\pa_{\m}\pa_{\f}S(\m,\f)|\geq C
\end{align*}
uniformly in $\m\in [0,1]$, $\f\in [0,1]$ and $\l\gg 1$. Suppose that $S(\cdot,\f)$ has a unique critical point~${\m_0(\f)}$ which is smooth with respect to $\f$ and that $\m_0(0)=0$. Then
 \smash{$|I(\l,\f)|\lesssim \l^{-\n-\frac{1}{2}}$} uniformly in $\l\gg 1$ and $\f\in [0,1]$.
\end{Proposition}

\begin{Remark}\quad
\begin{itemize}\itemsep=0pt
\item[(1)] Without the assumption $\m_0(0)=0$, we obtain \smash{$I(\l,\f)\lesssim \l^{-\frac{1}{2}}$} only. The additional decay~$\l^{-\n}$ comes from the factor $(1+\l \f\m)^{-\n}$ in the assumption of $\c$ and the fact that $\m_0(0)=0$ as is mentioned before.

\item[(2)] We observe that the critical point $\m_0(\f)$ of $S$ is close to $0$ if $\f\sim 0$. There we take advantage of the assumption that the symbol $\c$ vanishes at zero with order $2\n$ like $|\c(\m,\l,\f)|\lesssim \m^{2\n}$. On the other hand, if $\f$ is away from zero, then the critical point $\m_0(\f)$ is also away from zero. In this case, $\c$ itself decays like $\l^{-\n}$ near the critical point (due to $|\f|\gtrsim 1$ and~${|\m_0(\f)|\gtrsim 1}$), which makes us to prove the improved decay. The difficulty here is to deal with the case where $\f$ is small enough but depending on $\l$.

\item[(3)]
The critical point of $S(\cdot,\f)$ is unique due to the condition $\big|\pa_{\m}^2S(\m,\f)\big|\geq c$ in this case.
\end{itemize}
\end{Remark}

\begin{proof}[Proof of Proposition~\ref{Stphasemovecrit}]
We may assume \smash{$\supp \c(\cdot,\l,\f)\subset \bigl[\l^{-\frac{1}{2}},1\bigr]$}. In fact, we set
\begin{align*}
\c'(\m,\l,\f)=\chi\bigl(\l^{\frac{1}{2}}\m\bigr)\c(\m,\l,\f),
\end{align*}
where $\chi\in C^{\infty}([0,\infty),\R)$ satisfies $\chi(\m)=1$ on $\m\leq 1$ and $\chi(\m)=0$ for $\m\geq 2$. Using the bound~${|\c(\m,\l,\f)|\lesssim \m^{2\n}}$, we have
\begin{align*}
\left|\int_{\R}{\rm e}^{{\rm i}\l S(\m,\f)}\c'(\m,\l,\f){\rm d}\m\right|\lesssim \l^{-\n}\left|\int_{\R}\chi\bigl(\l^{\frac{1}{2}}\m\bigr){\rm d}\m\right|\lesssim \l^{-\n-\frac{1}{2}}.
\end{align*}
Hence, we can replace $\c$ by $\c-\c'$, where we note that $\c-\c'$ satisfies
\begin{align*}
\supp (\c-\c')(\cdot,\l,\f)\subset\bigl[\l^{-\frac{1}{2}},1\bigr],\qquad |\pa_{\m}^{\a}(\c-\c')(\m,\l,\f)|\leq C_{\a}\m^{2\n-\a}(1+\l \f\m)^{-\n}.
\end{align*}
In the following, we assume \smash{$\supp \c(\cdot,\l,\f)\subset \bigl[\l^{-\frac{1}{2}},1\bigr]$}.

 $(i)$ First, we consider the case $\f\in [0,1]$ satisfying \smash{$\m_0(\f)\leq 2^{-1}\l^{-\frac{1}{2}}$}.
Since the signature of~$\pa_{\m}^2S(\m,\f)$ does not change for \smash{$\m\in \bigl[\l^{-\frac{1}{2}},1\bigr]$} and $\f\in [0,1]$, we have
\[
|\pa_{\m}S(\m,\f)|=\left|\int_{\m_0(\f)}^{\m}\pa_{\m}^2S(\m',\f){\rm d}\m'\right|=\int_{\m_0(\f)}^{\m}\big|\pa_{\m}^2S(\m',\f)\big|{\rm d}\m'\geq c(\m-\m_0(\f))\geq 2^{-1}c\m
\]
 for \smash{$\m\in \bigl[\l^{-\frac{1}{2}},1\bigr]$}. By integrating by parts and using \smash{$\supp \c(\cdot,\l,\f)\subset \bigl[\l^{-\frac{1}{2}},1\bigr]$}, we have
\begin{align*}
|I(\l,\f)|=\left|(-{\rm i}\l)^{-N}\int_{\R}{\rm e}^{{\rm i}\l S(\m,\f)} L^N\c(\m,\l,\f){\rm d}\m\right|\leq \l^{-N}\int_{\l^{-\frac{1}{2}}}^1|L^N\c(\m,\l,\f)|{\rm d}\m,
\end{align*}
where $L=\pa_{\m}\circ (\pa_{\m}S)(\m,\f)^{-1}$ and we take $N>\n+1$. We observe from the assumption and Lemma \ref{Leibnizrule} that $\big|L^N\c(\m,\l,\f)\big|\lesssim \m^{2\n-2N}$, which leads to
\begin{align*}
|I(\l,\f)|\lesssim \l^{-N}\int_{\l^{-\frac{1}{2}}}^1\m^{2\n-2N}{\rm d}\m\lesssim \l^{-\n-\frac{1}{2}}.
\end{align*}

 $(ii)$ Next, we consider the case $\f\in [0,1]$ satisfying \smash{$\m_0(\f)\in \bigl[ 2^{-1}\l^{-\frac{1}{2}}, 1\bigr]$}.
Differentiating $(\pa_{\m}S)(\m_0(\f),\f)=0$ with respect to $\f$, we have
\begin{align*}
|\m_0'(\f)|=\bigl|-(\pa_{\m}\pa_{\f}S)(\m_0(\f),\f)/\pa_{\m}^2S(\m_0(\f),\f)\bigr|\sim 1,
\end{align*}
and hence $\m_0(\f)\sim \f$ for $\f\in [0,1]$ due to $\m_0(0)=0$. Now the assumption \smash{$\m_0(\f)\geq 2^{-1}\l^{-\frac{1}{2}}$} yields \smash{$\f\gtrsim \l^{-\frac{1}{2}}$}.

By scaling, we have
\begin{align*}
I(\l,\f)=\int_{\R}{\rm e}^{{\rm i}\l S(\m,\f)}\c(\m,\l,\f){\rm d}\m=\m_0(\f)\l^{-\n}\int_{\R}{\rm e}^{{\rm i}\l \m_0(\f)^2S_{\f}(\m)}\c_{\f}(\m,\l){\rm d}\m,
\end{align*}
where we set
\begin{align*}
S_{\f}(\m)=\m_0(\f)^{-2}S(\m_0(\f)\m,\f),\qquad \c_{\f}(\m,\l)=\l^{\n}\c(\m_0(\f)\m,\l,\f).
\end{align*}
We note that $\m_0(\f)$ is a unique critical point of $S(\m,\f)$.
Let $\g_1,\g_2,\g_3\in C^{\infty}(\R;[0,1])$ such that~${\g_1+\g_2+\g_3=1}$, $\supp \g_1\subset \bigl(-\infty,\frac{3}{4}\bigr]$, $\supp \g_2\subset \bigl[\frac{1}{2},\frac{3}{2}\bigr]$ and $\supp \g_3\subset \bigl[\frac{5}{4},\infty\bigr)$. We write
\begin{align*}
I(\l,\f)
=\m_0(\f)\l^{-\n}\sum_{j=1}^3\int_{\R}{\rm e}^{{\rm i}\l \m_0(\f)^2S_{\f}(\m)}\c_{\f}(\m,\l)\g_j(\m){\rm d}\m
=\m_0(\f)\l^{-\n}\sum_{j=1}^3I_j(\l,\f).
\end{align*}

First, we deal with $I_2(\l,\f)$.
Now we see $|\pa_{\m}S(\m,\f)|\geq c|\m-\m_0(\f)|$ for $\m\in [0,1]$ and $\f\in [0,1]$. Then we have
\begin{align*}
|\pa_{\m}S_{\f}(\m)|=\m_0(\f)^{-1}\cdot |(\pa_{\m}S)(\m_0(\f)\m,\f)|\geq c|\m-1|
\end{align*}
for $\m\in \supp \c_{\f}(\cdot,\l)$ and $\f\in (0,1]$.
Moreover, for $\a\in\mathbb{N}$ and $\a'\in \mathbb{N}\setminus\{0\}$, we have
\begin{gather*}
|\pa_{\m}^{\a}\c_{\f}(\m,\l)|=\l^{\n} \m_0(\f)^{\a}|(\pa_{\m}^{\a}\c)(\m_0(\f)\m,\l,\f)|\\
\phantom{|\pa_{\m}^{\a}\c_{\f}(\m,\l)|}{}\lesssim \l^{\n}\m_0(\f)^{\a} (\m_0(\f) \m)^{2\n-\a}\bigl(1+\l \m_0(\f)^2 \m\bigr)^{-\n}\lesssim \m^{\n-\a},\\
\supp (\c_{\f}(\cdot,\l))\subset \big\{\m_0(\f)^{-1}\l^{-\frac{1}{2}}\leq \m\leq \m_0(\f)^{-1}\big\},\qquad \big|\pa_{\m}^{1+\a'}S_{\f}(\m)\big|\lesssim |\m_0(\f)|^{\a'-1}
\end{gather*}
Thus the stationary phase theorem (see Lemma \ref{stphaselem} with $k=2$, $j=0$ and $x_0=1$) implies
\begin{align*}
\m_0(\f)\l^{-\n}|I_2(\l,\f)|
\lesssim \m_0(\f)\l^{-\n}\cdot \bigl(\l \m_0(\f)^2\bigr)^{-\frac{1}{2}}=\l^{-\frac{1}{2}-\n},
\end{align*}
where we use $\l \m_0(\f)^2\gtrsim 1$.

Next, we consider the estimate of $I_{1}(\l,\f)$. We note that $|\pa_{\m}S_{\f}(\m)|\gtrsim 1$ for $\m\in \supp \g_1$.
Then the non-stationary phase theorem (see Lemma \ref{singularasym} with $\m=\n$) implies
\begin{align*}
\m_0(\f)\l^{-\n}|I_1(\l,\f)|\lesssim \m_0(\f)\l^{-\n}\cdot \bigl(\l \m_0(\f)^2\bigr)^{-\n-1}\lesssim \f\l^{-\n}\cdot \bigl(\l \f^2\bigr)^{-\frac{1}{2}}=\l^{-\frac{1}{2}-\n},
\end{align*}
where we use $\l \m_0(\f)^2\gtrsim 1$.

Finally, we consider the term $I_{3}(\l,\f)$. We note that $|\pa_{\m}S_{\f}(\m)|\gtrsim (|\m|+1)$ for $\m\in \supp \g_3$. Then, by integrating by parts $N(>1)$ times, we have
\begin{align*}
\m_0(\f)\l^{-\n}|I_3(\l,\f)|\lesssim \m_0(\f)\l^{-\n}\cdot \bigl(\l \m_0(\f)^2\bigr)^{-N}\lesssim \m_0(\f)\l^{-\n}\cdot \bigl(\l \m_0(\f)^2\bigr)^{-\frac{1}{2}}=\l^{-\frac{1}{2}-\n},
\end{align*}
where we use $\l \m_0(\f)^2\gtrsim 1$. This completes the proof.
\end{proof}

\section[Asymptotic expansion of special functions and decomposition of the sum]{Asymptotic expansion of special functions\\ and decomposition of the sum}

\subsection{Asymptotic expansion of Bessel functions}

Here we discuss asymptotic behavior of Bessel functions. For our purpose, we need symbol-type estimates, which is the estimates for higher order derivatives of the amplitudes. Such estimates are studied in \cite[Theorem 9.1]{IJ} up to the second derivatives. However, their method cannot be applied to estimates for higher-derivatives. Their theorem is not sufficient for our purpose since we need estimates whose order depends on $\n$ here.
Now we write down the statement here and give its proof in Appendix~\ref{AppBessel}.

We define
\begin{align}\label{hdef}
h_1(z)=\sqrt{1-z^2}-z\cos^{-1} z,
\end{align}
where $\cos^{-1} z\in [0,\pi/2]$ for $0\leq z\leq 1$.

\begin{Proposition}\label{Besselasymp}\quad
\begin{itemize}\itemsep=0pt
\item[$(i)$] There are smooth functions $a_{\pm,y}\colon \bigl[0,y-\frac{1}{2}y^{\frac{1}{3}}\bigr]\to \mathbb{C}$ such that
\begin{align*}
J_{\m}(y)=y^{-\frac{1}{4}}(y-\m)^{-\frac{1}{4}}\bigl(a_{+,y}(\m){\rm e}^{{\rm i}yh_1(\frac{\m}{y})}+a_{-,y}(\m){\rm e}^{-{\rm i}yh_1(\frac{\m}{y})} \bigr)
\end{align*}
and for each $\a\in\mathbb{N}$, there exists $C_{\a}>0$ such that
$
|\pa_{\m}^{\a}a_{\pm,y}(\m)|\leq C_{\a}(y-\m)^{-\a}$, for $ y\geq 8$, \smash{$ \m\in \bigl[0,y-\frac{1}{2}y^{\frac{1}{3}}\bigr]$}.
\item[$(ii)$] Let $N>0$ and $\a\in\mathbb{N}$. Then there exist $C_{\a}>0$ and $C_{\a N}>0$ such that
\begin{align*}
|\pa_{\m}^{\a}J_{\m}(y)|\leq \begin{cases}C_{\a}y^{-\frac{1+\a}{3}} \qquad \text{for}\  y\geq 8,\  \m\in \bigl[y-2y^{\frac{1}{3}},y+2y^{\frac{1}{3}}\bigr],\\
C_{\a N}\m^{-\frac{1}{4}}(\m-y)^{-\frac{1}{4}-\a} \bigl(y^{-1}(\m-y)^3 \bigr)^{-N} \\
\qquad \text{for}\  y\geq 8,\  \m\in \bigl[y+\frac{1}{2}y^{\frac{1}{3}},\infty\bigr).
\end{cases}
\end{align*}
\end{itemize}
\end{Proposition}

\begin{Remark}\quad
\begin{itemize}\itemsep=0pt
\item[(1)] The estimate for \smash{$\m\in \bigl[y+\frac{1}{2}y^{\frac{1}{3}},\infty\bigr)$} is weaker than the one in \cite[Theorem 9.1]{IJ} at least for $\a=0,1,2$. In fact, \cite[Theorem 9.1]{IJ} shows that $J_{\m}(y)$ and its derivatives decay exponentially with respect to $y^{-1}(\m-y)^3$ (due to the term ${\rm e}^{-yh_2(\m/y)}$ in \cite{IJ}) although we just obtain the polynomial decay here.

\item[(2)] Another drawback of this proposition compared to \cite[Theorem 9.1]{IJ} is not to give symbol-type estimates of the $Y$-Bessel functions.
\end{itemize}
\end{Remark}

In particular, we have a uniform bound
\begin{align}\label{Besselunif}
|J_{\m}(y)|\leq Cy^{-\frac{1}{3}}\qquad y\geq 8,\quad \m\geq 0.
\end{align}

\subsection{Asymptotic expansion of Gegenbauer polynomials}

The following proposition is a refinement of \cite[Lemma 10.2]{IJ} and its proof is given in Appendix~\ref{appasymgegensubsec}.

\begin{Proposition}\label{Gegenasymp}
Let $\n\geq 0$.
There are functions $g_{+}(m,\f), g_{-}(m,\f)$ which are smooth respect to $m\in [1,\infty)$ and a function $r(m,\f)$ defined for $m\in\mathbb{N}^*$ such that
\begin{align*}
\n^{-1}C_m^{\n}(\cos\f)=\sum_{\pm}g_{\n,\pm}(m,\f){\rm e}^{\pm {\rm i}m\f}+r(m,\f)
\end{align*}
for $m\in \mathbb{N}^*$ $($we interpret the left-hand side for $\n=0$ as \smash{$\lim_{\n\searrow 0}\n^{-1}C_m^{\n}(\cos \f))$}, and
\begin{align*}
&|\pa_m^{\a}g_{\n,\pm}(m,\f)|\leq
C_{\a,\n}m^{2\n-1-\a}(1+m\sin\f)^{-\n},\qquad m\geq 1,\\
&|r(m,\f)|\leq C_{N,\n}m^{-N},\qquad m\in\mathbb{N}^*
\end{align*}
for each $N>0$, $\a\in\mathbb{N}$, and $\f\in [0,\pi]$ with constants $C_{\a,\n}, C_{N,\n}>0$.
\end{Proposition}

\begin{Remark}\quad
\begin{itemize}\itemsep=0pt
\item[(1)] In particular, for each $\e>0$, we have $|\pa_m^{\a}g_{\n,\pm}(m,\f)|\leq C_{\a,\n,\e}m^{\n-1-\a}$ uniformly in $\f\in [\e,\pi-\e]$.
\item[(2)] We do not impose that $g_{\n,\pm}$ are continuous with respect to $\f\in [0,\pi]$. In this paper, we just need uniform estimates of $g_{\n,\pm}$ in $\f\in [0,\pi]$.
\item[(3)]
In \cite[Lemma 10.2]{IJ}, a similar estimate for $\n^{-1}C_m^{\n}(\cos\f)$ divided by $C_m^{\n}(1)$ is given.
Here we also deal with the estimate for $C_m^{\n}(1)$.
Moreover, the ranges of the parameters are extended compared to there.
\end{itemize}
\end{Remark}

In particular, we have a uniform bound: For $\n\geq 0$, there exists $C>0$ such that
\begin{align}\label{Gegenunif}
\big|\n^{-1}C_m^{\n}(\cos \f)\big|\leq Cm^{2\n-1}\qquad m\in \mathbb{N}^*,\quad \f\in[0,\pi],
\end{align}
where the left-hand side is interpreted as $\big|\lim_{\n\searrow 0}\n^{-1}C_m^{\n}(\cos \f)\big|$ when $\n=0$.

\subsection{Decomposition of the sum}\label{subsecdecom}

In this subsection, we divide the sum $\sI(b,\nu;- {\rm i}y;\cos\f)$ into three parts according to the asymptotic behavior of the Bessel function.

We fix $b>0$ and set
\begin{align}\label{cbndef}
L_{b,\n}:=\Gamma(b\nu+1)2^{b\n}.
\end{align}
From \eqref{sIdefn}, we obtain the formula
\begin{align}\label{sumbesell}
\sI(b,\nu;- {\rm i}y;\cos\f)=&L_{b,\n}y^{-b\n}\sum_{m=0}^{\infty}{\rm e}^{-\frac{\pi}{2}bm{\rm i}}J_{b(m+\n)}(y)(m+\nu)\n^{-1}C^{\nu}_m(\cos\f).
\end{align}
In the following, we consider the case $y\gg 1$ for simplicity. The case $y\ll -1$ is similarly dealt with (see proof of Theorem \ref{keythm} in Section \ref{subsec:proof}).

We define
\begin{align}\label{muemudef}
\m(m):=b(m+\n).
\end{align}
Corresponding to the asymptotic behavior of the Bessel functions, we define
\begin{gather*}
\Omega_1=\left\{m\in[1,\infty)\mid 1\leq \m(m)\leq y-\frac{1}{2}y^{\frac{1}{3}}\right\},\\ \Omega_2=\big\{m\in[1,\infty)\mid y-2y^{\frac{1}{3}}\leq \m(m)\leq y+2y^{\frac{1}{3}}\big\},\\
\Omega_3=\left\{m\in[1,\infty)\mid \m(m)\geq y+\frac{1}{2}y^{\frac{1}{3}}\right\}.
\end{gather*}

\begin{Lemma}\label{Cutoff}
For $y\gg 1$,
there exist $\chi_0\in C_c^{\infty}(\R;[0,1])$ and $\chi_{j,y}\in C^{\infty}(\R;[0,1])$ $(1\le j\le 3)$ such that $\chi_0(m)=1$ for $\m(m)\leq 2$ or $m\leq 1$ and
\begin{align}
&\chi_0(m)+\sum_{j=1}^3\chi_{j,y}(m)=1\qquad \text{ for $m\geq 0$},\qquad \supp \chi_{j,y}\subset \Omega_j,\qquad |\pa_m^{\a}\chi_{2,y}(m)|\leq C_{\a}y^{-\frac{\a}{3}},\nonumber\\
&|\pa_m^{\a}\chi_{1,y}(m)|\leq C_{\a}\max(m^{-\a}, (y-\m(m))^{-\a}),\qquad |\pa_m^{\a}\chi_{3,y}(m)|\leq C_{\a}(\m(m)-y)^{-\a}\label{POUproper}.
\end{align}
\end{Lemma}

\begin{proof}
Let $\chi_0\in C_c^{\infty}(\R;[0,1])$ and $\chi_j\in C^{\infty}(\R;[0,1])$ for $j=1,2,3$ such that $\chi_0(m)=1$ for $\m(m)\leq 2$ or $m\leq 1$, \smash{$\sum_{j=1}^3\chi_j=1$} on $\R$, $\supp \chi_1\subset \bigl(-\infty,-\frac{1}{2}\bigr]$, $\supp \chi_2\subset [-2,2]$ and $\supp \chi_3\subset \bigl[\frac{1}{2},\infty\bigr)$. We define
\begin{align*}
\chi_{j,y}(m)=(1-\chi_0(m))\chi_j\bigl(y^{-\frac{1}{3}}(\m(m)-y)\bigr)
\end{align*}
for $j=1,2,3$.
They satisfies the desired properties.
\end{proof}

For $\s=(\s_1,\s_2)\in \{\pm\}\times \{\pm\}$ and $j=2,3$, we set
\begin{align*}
&S_{1,\s}(m,y,\f)=\s_1yh_1\left(\frac{\m(m)}{y}\right) +\left(\s_2\f-\frac{\pi}{2}b\right)m,\qquad S_{\s_2}(m,y,\f)=\left(\s_2\f-\frac{\pi}{2}b\right)m,
\end{align*}
where $h_1$ is defined in \eqref{hdef}. We remark that $S_{1,\sigma}(m,y,\f)$ is defined for $0\le \m(m)\le y$.
We take~${y\gg 1}$ such that $\Omega_2\cup \Omega_3\subset \{\m(m)\geq 8\}$ in order to apply Proposition \ref{Besselasymp}.
We define
\begin{align*}
&\z_{1,\s}(m,y,\f)=L_{b,\n}y^{-b\n-\frac{1}{4}}\chi_{1,y}(m)(m+\n)(y-\m(m))^{-\frac{1}{4}}g_{\n,\s_2}(m,\f)a_{\s_1,y}(\m(m)),\\
&\z_{j,\s_2}(m,y,\f)=L_{b,\n}y^{-b\n}\chi_{j,y}(m)(m+\n)J_{\m(m)}(y)g_{\n,\s_2}(m,\f),\qquad j=2,3,
\end{align*}
where $a_{\pm,y}$ and $g_{\n,\pm}$ are defined in Propositions \ref{Besselasymp} and \ref{Gegenasymp} respectively and $L_{b,\n}$ is a constant defined in~\eqref{cbndef}.

Now it follows from Propositions \ref{Besselasymp}, \ref{Gegenasymp}, the formula \eqref{sumbesell} and Lemma \ref{Cutoff} that
\begin{align}\label{sIdecom}
\sI(b,\nu;- {\rm i}y;\cos\f)=\sum_{\s\in \{\pm\}\times \{\pm\}}I_{1,\s}(y,\f)+\sum_{j=2}^3\sum_{\s_2\in \{\pm\}}I_{j,\s_2}(y,\f)+R(y,\f),
\end{align}
where we set
\begin{gather*}
I_{1,\s}(y,\f):=\sum_{m=1}^{\infty}\z_{1,\s}(m,y,\f){\rm e}^{{\rm i}S_{1,\s}(m,y,\f)},\qquad I_{j,\s_2}(y,\f):=\sum_{m=1}^{\infty}\z_{j,\s_2}(m,y,\f){\rm e}^{{\rm i}S_{\s_2}(m,y,\f)},\\
R(y,\f):=L_{b,\n}y^{-b\n}\sum_{m=0}^{\infty}\chi_0(m){\rm e}^{-\frac{\pi}{2}bm{\rm i}}J_{\m(m)}(y)(m+\nu)\n^{-1}C^{\nu}_m(\cos\f)\\
\phantom{R(y,\f):=}{}+L_{b,\n}y^{-b\n}\sum_{m=0}^{\infty}
(1-\chi_0(m))
{\rm e}^{-\frac{\pi}{2}bm{\rm i}}J_{\m(m)}(y)(m+\nu)r(m,\f)
\end{gather*}
for $j=2,3$.
The remainder term $R(y,\f)$ is easy to handle.

\begin{Proposition}\label{Ryphiprop}
There exists $C>0$ such that
\begin{align*}
|R(y,\f)|\leq Cy^{-b\n-\frac{1}{3}}
\end{align*}
for $y\gg 1$ and $\f\in [0,\pi]$.
\end{Proposition}
\begin{proof}
Since $\chi_0$ is compactly supported, then the first term in the definition of $R(y,\f)$ is bounded by a constant times \smash{$y^{-b\n}\cdot y^{-\frac{1}{3}}=y^{-b\n-\frac{1}{3}}$} due to \eqref{Besselunif} and \eqref{Gegenunif}. Moreover, the second term has a similar bound since $r(m,y)$ is rapidly decreasing by Proposition $\ref{Gegenasymp}$.
\end{proof}

In the following, we focus on studying $I_{1,\s}$ and $I_{j,\s_2}$.

\subsection{Properties of the phase functions}

Here we prove some estimates of the phase functions defined in the last subsection.
We recall $h_1(z)=\sqrt{1-z^2}-z\cos^{-1} z$ and
$
h_1'(z)=-\cos^{-1}z$.
For $\s=(\s_1,\s_2)\in \{\pm\}\times \{\pm\}$ and $j=2,3$, we have
\begin{align*}
\pa_mS_{1,\s}(m,y,\f)=-\s_1 b\cos^{-1}\left(\frac{\m(m)}{y} \right)+\s_2\f-\frac{\pi}{2}b,\qquad \pa_mS_{\s_2}(m,y,\f)=\s_2 \f-\frac{\pi}{2}b,
\end{align*}
where $\m(m)$ is defined in \eqref{muemudef}.

\begin{Lemma}\label{phaseest}
Let $0<\d<1$ and $\s\in \{\pm\}\times \{\pm\}$.
\begin{itemize}\itemsep=0pt
\item[$(i)$] For $\a\in\mathbb{N}^*$, there exists $C_{\a}>0$ such that
\begin{align*}
\big|\pa_m^{\a+1}S_{1,\s}(m,y,\f)\big|\leq \begin{cases}C_{\a}y^{-\a}&  \text{for}\  1\leq \m(m)\leq \d y,\\
C_{\a}y^{-\frac{1}{2}}(y-\m(m))^{-\a+\frac{1}{2}}& \text{for}\  \d y\leq \m(m)\leq y-\frac{1}{2}y^{\frac{1}{3}}
\end{cases}
\end{align*}
and for $m\in \Omega_1$, $y\gg 1$ and $\f\in [0,\pi]$.
\item[$(ii)$] For $\a\in\mathbb{N}^*\setminus \{1\}$, we have $\pa_m^{\a}S_{\s_2}(m,y,\f)=0$.
\end{itemize}
\end{Lemma}

\begin{proof}
Since \smash{$\pa_z\cos^{-1}z=-(1-z^2)^{-\frac{1}{2}}$}, we have \smash{$\pa_{m}^2S_{1,\s}(m,y,\f)=\s_1b^2\bigl(y^2-\m(m)^2\bigr)^{-\frac{1}{2}}$}.
We have
\begin{align*}
\big|\pa_{m}^{\a+1}S_{1,\s}(m,y,\f)\big|={}&\big|\pa_{m}^{\a-1}\pa_{m}^2S_{1,\s}(m,y,\f)\big|\lesssim y^{\a-1}\bigl(y^2-\m(m)^2\bigr)^{-\a+\frac{1}{2}}\\
\le{}&y^{-\frac{1}{2}}(y-\m(m))^{-\a+\frac{1}{2}}.
\end{align*}
If $1\leq \m(m)\leq\d y$, then we have $(y-\m(m))^{-1}\lesssim y^{-1}$. Combining these estimates, we obtain the part (i).
The part (ii) is easy to prove.
\end{proof}

\subsection{Estimates of the amplitudes}

In this subsection, we deduce estimates for $\z_{1,\s}$ and $\z_{j,\s_2}$ which are defined in Section \ref{subsecdecom}. Recall from \eqref{muemudef} that $\m(m)=b(m+\n)$.

\begin{Lemma}\label{ampesti}

Let $\s=(\s_1,\s_2)\in \{\pm\}\times\{\pm\}$, $N>0$, $0<\e<\pi/2$ and $\a\in\mathbb{N}$. Then
\begin{itemize}\itemsep=0pt
\item[$(i)$] $\supp \z_{1,\s}(\cdot,y,\f)\subset \Omega_1$ and $\supp \z_{j,\s_2}(\cdot,y,\f)\subset \Omega_j$ for $j=2,3$.
\item[$(ii)$] Let $0<\d<1$. We have
\begin{align*}
|\pa_{m}^{\a}\z_{1,\s}(m,y,\f)|\lesssim \begin{cases}y^{-b\n-\frac{1}{2}}(1+m)^{2\n-\a}(1+m\sin \f)^{-\n} \qquad \text{when}\  \m(m)\leq \d y, \\
y^{(2-b)\n-\frac{1}{4}} (y-\m(m))^{-\frac{1}{4}-\a}(1+m\sin \f)^{-\n}\\
\qquad\text{when}\  \d y\leq \m(m)\leq y-\frac{1}{2}y^{\frac{1}{3}}
\end{cases}
\end{align*}
for $m\geq 1$, $y\gg 1$ and $\f\in [0,\pi]$. In particular, for fixed $\e>0$, we have
\begin{align*}
|\pa_{m}^{\a}\z_{1,\s}(m,y,\f)|\lesssim \begin{cases}y^{-b\n-\frac{1}{2}}(1+m)^{\n-\a}&  \text{when}\  \m(m)\leq \d y,\\
y^{(1-b)\n-\frac{1}{4}} (y-\m(m))^{-\frac{1}{4}-\a}& \text{when}\  \d y\leq \m(m)\leq y-\frac{1}{2}y^{\frac{1}{3}}
\end{cases}
\end{align*}
for $m\geq 1$, $y\gg 1$ and $\f\in [\e,\pi-\e]$.
\item[$(iii)$] We have
\begin{align*}
|\z_{2,\s_2}(m,y,\f)|\lesssim \begin{cases}y^{(1-b)\n-\frac{1}{3}}& \text{for}\  \f\in [\e,\pi-\e],\\
y^{(2-b)\n-\frac{1}{3}}& \text{for}\  \f\in [0,\pi]
\end{cases}
\end{align*}
and for $m\geq 1$, $y\gg 1$. Moreover,
\begin{align*}
|\pa_{m}^{\a}\z_{2,\s_2}(m,y,\f)|\lesssim y^{(2-b)\n-\frac{1+\a}{3}}
\end{align*}
for $m\geq 1$, $y\gg 1$ and $\f\in [0,\pi]$.
\item[$(iv)$]
We have
\begin{align*}
|\z_{3,\s_2}(m,y,\f)|\lesssim \begin{cases}y^{(2-b)\n+\frac{1}{12}}(\m(m)-y)^{-\frac{5}{4}}\\
\qquad \text{for}\  y+\frac{1}{2}y^{\frac{1}{3}}\leq \m(m)\leq 4y,\  \f\in [0,\pi],\\
y^{(1-b)\n+\frac{1}{12}}(\m(m)-y)^{-\frac{5}{4}}\\
\qquad \text{for}\  y+\frac{1}{2}y^{\frac{1}{3}}\leq \m(m)\leq 4y,\  \f\in [\e,\pi-\e],\\
(1+m)^{-N}\qquad \text{for}\  \m(m)\geq 2y,\  \f\in [0,\pi]
\end{cases}
\end{align*}
and for $m\geq 1$, $y\gg 1$. Moreover,
\begin{align*}
|\pa_{m}^{\a}\z_{3,\s_2}(m,y,\f)|\lesssim y^{(2-b)\n-\frac{1}{4}-\frac{\a}{3}}
\end{align*}
for $m\geq 1$ $y\gg 1$ and $\f\in [0,\pi]$.
\end{itemize}
\end{Lemma}

\begin{proof}
 $(i)$ This follows from the definition of $\z_{1,\s}$, $\z_{j,\s_2}$ and the support properties of $\chi_{j,y}$ \eqref{POUproper}.

$(ii)$ It turns out from Lemmas \ref{Besselasymp}, \ref{Gegenasymp} and \ref{Cutoff} that
\begin{align*}
|\pa_{m}^{\a}\z_{1,\s}(m,y,\f)|\lesssim {}&y^{-b\n-\frac{1}{4}}(1+m)^{2\n}(y-\m(m))^{-\frac{1}{4}}(1+m\sin \f)^{-\n}\\
&\times \max((y-\m(m))^{-\a}, m^{-\a}).
\end{align*}
Now the claim in (ii) follows from the support property of $\z_{1,\s}$.

 $(iii)$ We note that \smash{$\Omega_2=\big\{m\in\R\mid y-2y^{\frac{1}{3}}\leq \m(m)\leq y+2y^{\frac{1}{3}}\big\}$}. This part directly follows from Lemmas \ref{Besselasymp}, \ref{Gegenasymp} and \ref{Cutoff}.

 $(iv)$ We recall $\m(m)=b(m+\n)$ and that $\z_{3,\s_2}(m,y,\f)$ is supported in $\Omega_3=\big\{m\in\R\mid \m(m)\geq y+\smash{\frac{1}{2}y^{\frac{1}{3}}}\big\}$.
By Lemmas \ref{Besselasymp}, \ref{Gegenasymp} and \ref{Cutoff}, for each $N>0$ and $\a\in\mathbb{N}$, we have
\begin{align*}
|\pa_{m}^{\a}\z_{3,\s_2}(m,y,\f)|\lesssim y^{-b\n}(1+m)^{2\n-\frac{1}{4}}(1+m\sin \f)^{-\n}(\m(m)-y)^{-\frac{1}{4}-\a}\bigl(y^{-1}(\m(m)-y)^3\bigr)^{-N}
\end{align*}
for $m\geq 1$, $y\gg 1$ and $\f\in [0,\pi]$. We note that $y^{-1}(\m(m)-y)^3\geq 1/8$ and \smash{$\m(m)-y\geq \frac{1}{2}y^{\frac{1}{3}}$} for~${m\in \Omega_3}$. All the estimates we desire are proved by these inequalities.
\end{proof}

\section[Estimates of I\_\{2,s\_2\} and I\_\{3,s\_2\}]{Estimates of $\boldsymbol{I_{2,\s_2}}$ and $\boldsymbol{I_{3,\s_2}}$ }

In this section, we prove bounds for $I_{2,\s_2}$ and $I_{3,\s_2}$ appearing in \eqref{sIdecom}. They are easier to handle than $I_{1,\s}$. In this section, we assume $b>0$, $\n\geq 0$ and let $\s_2\in \{\pm\}$.

\subsection{Intermediate region}

\begin{Proposition}\label{I_2prop}
We have
\begin{align*}
|I_{2,\s_2}(y,\f)|\lesssim \begin{cases}
y^{(1-b)\n}& \text{if}\  b\notin 2\mathbb{Z}, \\
y^{(2-b)\n}& \text{if}\  b\in 2\mathbb{Z}
\end{cases}
\end{align*}
for $y\gg 1$ and $\f\in [0,\pi]$.
\end{Proposition}

\begin{proof}
By Lemma \ref{ampesti}\,(i) and (iii), we have \smash{$|\z_{2,\s_2}(m,y,\f)|\lesssim y^{(2-b)\n-\frac{1}{3}}$} and
\begin{align*}
|I_{2,\s_2}(y,\f)|\lesssim y^{(2-b)\n-\frac{1}{3}}\sum_{m\in \Omega_2\cap\mathbb{N}}1\lesssim y^{(2-b)\n},
\end{align*}
where we recall \smash{$\Omega_2=\big\{m\in\R\mid y-2y^{\frac{1}{3}}\leq \m(m)\leq y+2y^{\frac{1}{3}}\big\}$} and use the number of element of~${\Omega_2\cap \mathbb{N}}$ is bounded by \smash{$y^{\frac{1}{3}}$} times a constant.

We next consider the case $b\notin 2\mathbb{Z}$.
A similar argument implies $|I_{2,\s_2}(y,\f)|\lesssim y^{(1-b)\n}$ uniformly in $\f\in [\e,\pi-\e]$ (with a constant $0<\e<\pi/2$), where we use the first estimate in Lemma \ref{ampesti}\,(iii). Thus, it remains to prove the bound of $I_{2,\s_2}(y,\f)$ for $\f\in [0,\e]\cup [\pi-\e,\pi]$ with $\e>0$ small enough when $b\notin 2\mathbb{Z}$.
In this case, we have
$
-\frac{\pi}{2}b$, $ \s_2\pi-\frac{\pi}{2}b\notin 2\pi \mathbb{Z}$.
Then we find $\e>0$ small enough such that there is $c_{\e}>0$ satisfying $\mathrm{dist}\bigl(\s_2\f-\frac{\pi}{2}b,2\pi\mathbb{Z}\bigr)\geq c_{\e}$ for $\f\in [0,\e)\cup (\pi-\e,\pi]$. This implies $\mathrm{dist}(\pa_mS_{\s_2}(m,y,\f), 2\pi\mathbb{Z})\geq c_{\e}$ for $\f\in [0,\e)\cup (\pi-\e,\pi]$. Moreover, it follows from Lemma \ref{ampesti}\,(iii) that \smash{$|\pa_{m}^{\a}\z_{2,\s_2}(m,y,\f)|\lesssim y^{(2-b)\n-\frac{1}{3}-\frac{\a}{3}}$}. Clearly, we have $\pa_m^{\a+1}S_{\s_2}(m,y,\f)=0$ for $\a\geq 1$.
Applying Proposition \ref{discsum1} with $k=(2-b)\n-\frac{1}{3}$, $r=c_{\e}$, $M \sim y$ and $\rho=\frac{1}{3}$, for each~${N>0}$, we obtain
$
|I_{2,\s_2}(y,\f)|\lesssim y^{-N}$ for $\f\in [0,\e)\cup (\pi-\e,\pi]$.
This completes the proof.
\end{proof}

\subsection{Decaying region}

\begin{Proposition}\label{I_3prop}
We have
\begin{align*}
|I_{3,\s_2}(y,\f)|\lesssim \begin{cases}
y^{(1-b)\n}& \text{if}\  b\notin 2\mathbb{Z}, \\
y^{(2-b)\n}&\text{if}\  b\in 2\mathbb{Z}
\end{cases}
\end{align*}
for $y\gg 1$ and $\f\in [0,\pi]$.
\end{Proposition}

\begin{proof}We recall $\m(m)=b(m+\n)$. Taking $\chi\in C^{\infty}(\R;[0,1])$ such that $\chi(\m)=1$ for $\m\leq 2$ and~${\chi(\m)=0}$ for $\m\geq 4$ and setting $\overline{\chi}=1-\chi$, we write
\begin{align*}
I_{3,\s_2}(y,\f)&= \left(\sum_{m=1}^{\infty}(\chi(\m(m)/y)+\overline{\chi}(\m(m)/y)) \z_{3,\s_2}(m,y,\f){\rm e}^{{\rm i}S_{\s}(m,y,\f)} \right)\\
&=: I_{3,1,\s_2}(y,\f)+I_{3,2,\s_2}(y,\f).
\end{align*}

First, we deal with the second term $I_{3,2,\s_2}(y,\f)$. Let $N>0$. Then Lemma \ref{ampesti}\,(iv) implies
\begin{align*}
|\z_{3,\s_2}(m,y,\f)|\lesssim (1+m)^{-N-1}\qquad \text{for}\quad m\in \supp \overline{\chi}(\m(m)/y).
\end{align*}
Thus we obtain
\begin{align*}
|I_{3,2,\s_2}(y,\f)|\lesssim \sum_{\m(m)\geq 2y,m\geq 1}^{\infty}(1+m)^{-N-1}\lesssim y^{-N}.
\end{align*}

Next, we deal with the first term $I_{3,1,\s_2}(y,\f)$. Lemma \ref{ampesti}\,(iv) implies
\begin{align}\label{zeta32}
|\z_{3,\s_2}(m,y,\f)|\lesssim y^{(2-b)\n+\frac{1}{12}}(\m(m)-y)^{-\frac{5}{4}}\qquad \text{for}\quad m\in \supp \chi(\m(m)/y).
\end{align}
Hence,
\begin{align*}
|I_{3,1,\s_2}(y,\f)|\lesssim{}& y^{(2-b)\n+\frac{1}{12}}\sum_{\m(m)\in [y+\frac{1}{2}y^{\frac{1}{3}},4y]}(\m(m)-y)^{-\frac{5}{4}}\\
\lesssim{}& y^{(2-b)\n+\frac{1}{12}}\int_{y+\frac{1}{2}y^{\frac{1}{3}}-b}^{4y}(\m-y)^{-\frac{5}{4}}{\rm d}\m
\lesssim y^{(2-b)\n},
\end{align*}
where we use Lemma \ref{monotonesum} in the second inequality.
Consequently, we obtain $|I_{3,1,\s_2}(y,\f)|\lesssim y^{(2-b)\n}$, which completes the proof for $b\in 2\mathbb{Z}$. When $b\notin 2\mathbb{Z}$, a similar argument implies $|I_{3,1,\s_2}(y,\f)|\lesssim y^{(1-b)\n}$ uniformly in $ \f\in [\e,\pi-\e]$ (with a constant $0<\e<\pi/2$), where we use the second estimate in Lemma \ref{ampesti}\,(iv) instead of \eqref{zeta32}. Thus, it remains to prove the bound of~${I_{3,1,\s_2}(y,\f)}$ when $b\notin 2\mathbb{Z}$ for $\f\in [0,\e]\cup [\pi-\e,\pi]$ with $\e>0$ small enough.

Finally, we suppose $b\notin 2\mathbb{Z}$ and prove the bound for $I_{3,1,\s_2}(y,\f)$ for $\f\in [0,\e]\cup [\pi-\e,\pi]$ with~${\e>0}$ small enough. In this case, we have
$
-\frac{\pi}{2}b$, $ \s_2\pi-\frac{\pi}{2}b\notin 2\pi \mathbb{Z}$.
Then we find $\e>0$ small enough such that there is $c_{\e}>0$ satisfying $\mathrm{dist}(\s_2\f-\frac{\pi}{2}b,2\pi\mathbb{Z})\geq c_{\e}$ for $\f\in [0,\e)\cup (\pi-\e,\pi]$. This implies $\mathrm{dist}(\pa_mS_{3,\s_2}(m,y,\f), 2\pi\mathbb{Z})\geq c_{\e}$ for $\f\in [0,\e)\cup (\pi-\e,\pi]$. Moreover, it follows from Lemma \ref{ampesti}\,(iv) that \smash{$|\pa_{m}^{\a}\z_{3,1,\s_2}(m,y,\f)|\lesssim y^{(2-b)\n-\frac{1}{4}-\frac{\a}{3}}$}. Clearly, we have $\pa_m^{\a+1}S_{3,\s_2}(m,y,\f)=0$ for $\a\geq 1$.
Now applying Proposition \ref{discsum1} with $k=(2-b)\n-\frac{1}{4}$, $r=c_{\e}$, $M\sim y$ and $\rho=\frac{1}{3}$, for each $N>0$, we obtain
$
|I_{3,1,\s_2}(y,\f)|\lesssim y^{-N}$ for $\f\in [0,\e)\cup (\pi-\e,\pi]$.
This completes the proof.
\end{proof}

\section[Estimates of I\_\{1,s\}]{Estimates of $\boldsymbol{ I_{1,\s}}$}

In this section, we prove estimates for $I_{1,\s}$, which are more delicate than those of $I_{2,\s_2}$, $I_{3,\s_2}$.
We assume $b>0$, $\n\geq 0$ and let $\s\in \{\pm\}\times \{\pm\}$.

\subsection[General bounds for b>0]{General bounds for $\boldsymbol{b>0}$}

\begin{Proposition}\label{I_1prop}
Let $\e>0$. Then
\begin{align*}
|I_{1,\s}(y,\f)|\lesssim \begin{cases}
y^{(1-b)\n}& \text{for}\  \f\in [\e,\pi-\e], \\
y^{(2-b)\n}& \text{for}\  \f\in [0,\pi]
\end{cases}
\end{align*}
and for $y\gg 1$.
\end{Proposition}

\begin{proof}
We deal with the case $\f\in [0,\pi]$ only since the proof is almost same if we use the first estimate of Lemma \ref{ampesti}\,(ii) instead of the second one.

The identity \eqref{discon} implies
\begin{align}
I_{1,\s}(y,\f)={}&\int_{\R}\z_{1,\s}(m,y,\f){\rm e}^{{\rm i}S_{1,\s}(m,y,\f)}{\rm d}m\nonumber\\
&-\frac{1}{2\pi {\rm i}}\sum_{q\in\mathbb{Z}\setminus \{0\}}\frac{1}{q}\int_{\R}\tilde{\z}_{1,\s}(m,y,\f){\rm e}^{{\rm i}S_{1,\s}(m,y,\f)+2\pi {\rm i}qm}{\rm d}m,\label{I1sigmaint}
\end{align}
where we set
$
\tilde{\z}_{1,\s}(m,y,\f)=\pa_m\z_{1,\s}(m,y,\f)+{\rm i}(\pa_mS_{1,\s})(m,y,\f)\z_{1,\s}(m,y,\f)$.
By Lemmas \ref{phaseest} and \ref{ampesti}\,(ii), for each $\z\in \big\{\z_{1,\s},\tilde{\z}_{1,\s}\big\}$ we have
\begin{align}\label{sevzetaest}
|\pa_{m}^{\a}\z(m,y,\f)|\lesssim \begin{cases} y^{-b\n-\frac{1}{2}}(1+m)^{2\n-\a} & \text{when}\  \m(m)\leq \frac{3}{4}y, \\
y^{(2-b)\n-\frac{1}{4}} (y-\m(m))^{-\frac{1}{4}-\a}& \text{when}\  \frac{3}{4}y\leq \m(m)\leq y-\frac{1}{2}y^{\frac{1}{3}}
\end{cases}
\end{align}
for $m\geq 1$, $y\gg1$ and $\f\in [0,\pi]$. Moreover, $\supp \z\subset \Omega_1$ holds.

Now we consider an integral
\[
I_q:=\int_{\R}\z(m,y,\f){\rm e}^{{\rm i}S_{1,\s}(m,y,\f)+2\pi {\rm i}qm}{\rm d}m
\]
for $q\in\mathbb{Z}$ and $\z\in \{\z_{1,\s},\tilde{\z}_{1,\s}\}$. Taking $c_0>0$ such that $|\pa_mS_{1,\s}(m,y,\f)|\leq c_0$ for $m\geq 1$, $y\gg 1$ and $\f\in [0,\pi]$. Now we show that
\begin{align*}
|I_q|\lesssim \begin{cases}
y^{(2-b)\n} & \text{for}\  |q|\leq c_0/\pi,\\
y^{(2-b)\n}(1+|q|)^{-1}& \text{for}\  |q|> c_0/\pi,
\end{cases}
\end{align*}
and for $y\gg 1$ and $\f\in [0,\pi]$.
These estimates immediately imply the bound $|I_{1,\s}(y,\f)|\lesssim y^{(2-b)\n}$ due to the identity \eqref{I1sigmaint}.

First, we deal with the case $|q|> c_0/\pi$. Since $|\pa_m(S_{1,\s}(m,y,\f)+2\pi qm)|\geq 2\pi|q|-c_0\gtrsim (1+|q|)$, for $\n>0$, the integration by parts yields
\begin{align*}
|I_{q}|={}&\left|\int_{\R}\pa_m\left(\frac{1}{\pa_m(S_{1,\s}(m,y,\f)+2\pi qm) } \z(m,y,\f)\right) {\rm e}^{{\rm i}S_{1,\s}(m,y,\f)+2\pi {\rm i}qm}{\rm d}m \right| \\
\leq{}&\int_{\R}\left(\frac{|\pa_m^2S_{1,\s}(m,y,\f)||\z(m,y,\f)|}{|\pa_m(S_{1,\s}(m,y,\f)+2\pi qm)|^2 } +\frac{|\pa_m\z(m,y,\f)|}{|\pa_m(S_{1,\s}(m,y,\f)+2\pi qm)|} \right) {\rm d}m \\
\lesssim{}& (1+|q|)^{-1} y^{-b\n-\frac{1}{2}}\int_{1\leq\m(m)\leq \frac{3}{4}y,m\geq 1}
(1+m)^{2\n-1}{\rm d}m \\
&+(1+|q|)^{-1}y^{(2-b)\n-\frac{1}{4}}\int_{\frac{3}{4}y\leq\m(m)\leq y-\frac{1}{2}y^{\frac{1}{3}}}(y-\m(m))^{-\frac{5}{4}}{\rm d}m\lesssim (1+|q|)^{-1}y^{(2-b)\n},
\end{align*}
where we use \eqref{sevzetaest}, Lemma \ref{phaseest}\,(i) and the support property $\supp\z\subset \Omega_1$.
The case $\n=0$ can be proved similarly if we use the integration by parts twice.

Next, we consider the case $|q|\leq c_0/\pi$. By the change of variable $\m(m)(=bm+b\n)=y\m$,
\begin{align*}
I_q=&b^{-1}{\rm e}^{-{\rm i}(\s_2\frac{\f}{b}-\frac{\pi}{2}+\frac{2\pi q}{b})b\n} y^{(2-b)\n+\frac{1}{2}}\int_{\R}{\rm e}^{{\rm i}yS_q(\m,\f) }\c_{y,\f}(\m) {\rm d}\m,
\end{align*}
where we set
\begin{align*}
S_q(\m,\f)=\s_1h_1(\m) +\left(\s_2\frac{\f}{b}-\frac{\pi}{2}+\frac{2\pi q}{b}\right)\m,\qquad \c_{y,\f}(\m)=y^{(b-2)\n+\frac{1}{2}}\z\left(\frac{y\m}{b}-\n,y,\f\right).
\end{align*}
Thus it remains to show
\begin{align}\label{I_1stphase}
\left|\int_{\R}{\rm e}^{{\rm i}yS_q(\m,\f) }\c_{y,\f}(\m) {\rm d}\m\right|\lesssim y^{-\frac{1}{2}}\qquad \text{for}\quad y\gg 1,\quad \f\in [0,\pi].
\end{align}
By \eqref{sevzetaest},
\begin{align*}
|\pa_\m^{\a}\c_{y,\f}(\m)|\lesssim \m^{2\n-\a}(1-\m)^{-\frac{1}{4}-\a},\qquad \supp \c_{y,\f}\subset \big\{\m\in\R\mid 0\leq \m\leq 1-\frac{1}{2}y^{-\frac{2}{3}} \big\}
\end{align*}
for $y\gg 1$ and $\f\in [0,\pi]$.

We write
\begin{align*}
\int_{\R}{\rm e}^{{\rm i}yS_q(\m,\f) }\c_{y,\f}(\m){\rm d}\m&=\int_{\R}{\rm e}^{{\rm i}yS_q(\m,\f) }\c_{y,\f}(\m)\g_1(\m){\rm d}\m+\int_{\R}{\rm e}^{{\rm i}yS_q(\m,\f) }\c_{y,\f}(\m)\g_2(\m){\rm d}\m\\
&=I_1+I_2,
\end{align*}
where $\g_1,\g_2\in C_c^{\infty}(\R;[0,1])$ satisfy $\g_1(\m)+\g_2(\m)=1$ for $\m\in [0,1]$, $\g_1(\m)=1$ for $0\leq \m\leq 1-2\d$ and $\g_2(\m)=1$ for $1-\d\leq \m\leq 1$ where $\d>0$ is determined later.
Since
\begin{align*}
\big|\pa_{\m}^2S_{q}(\m,\f)\big|=\big|\s_1\pa_{\m}^2h_1(\m)\big|=\big|\bigl(1-\m^2\bigr)^{-\frac{1}{2}}\big|\ge 1\qquad \text{for}\quad \m\in \supp \g_1\cap[0,1],
\end{align*}
the stationary phase theorem (see Lemma \ref{stphaselem2}\,(i) with $\l=y$) implies $|I_1|\lesssim y^{-\frac{1}{2}}$.
On the other hand, using the change of variable $\m'=\sqrt{1-\m}$ (with $\m=1-\m'^2$), we have ${\rm d}\m=-2\m'{\rm d}\m'$ and~\smash{$
I_2=\int_{\R}{\rm e}^{{\rm i}y\tilde{S}_q(\m',\f) }\c_2(\m'){\rm d}\m'$},
where we set $\tilde{S}_q(\m',\f)=S_q\bigl(1-\m'^2,\f\bigr)$ and $\c_2(\m')=2\m'\c_{y,\f}\bigl(1-\m'^2\bigr)\g_2\bigl(1-\m'^2\bigr)$. We note that
\begin{align*}
|\pa_{\m'}^{\a}\c_2(\m')|\lesssim\m'^{\frac{1}{2}-\a},\qquad\supp\c_2\subset \left\{ \frac{1}{\sqrt{2}}y^{-\frac{1}{3}}\leq \m'\leq \sqrt{2\d}\right\}
\end{align*}
for $y\gg 1$ and $\f\in [0,\pi]$ and that $\tilde{S}_q$ is smooth with respect to $\m'$ close to $0$ and $\f\in[0,\pi]$.
It follows from the identity
\begin{align*}
\pa_{\m'}^3\tilde{S}_q(\m',\f)={}&\s_1\pa_{\m'}^3\bigl(h_1\bigl(1-\m'^2\bigr)\bigr)=12\s_1\m' h_1''\bigl(1-\m'^2\bigr)-8\s_1\m'^3h_1^{(3)}\bigl(1-\m'^2\bigr)\\
={}&4\sqrt{2}\s_1+O(\m')\qquad \text{as}\quad \m'\to 0
\end{align*}
that \smash{$\big|\pa_{\m'}^3\tilde{S}_{q}(\m',\f)\big|\gtrsim 1$} for $\m'\in \supp \c_2$ and $\f\in [0,\pi]$ if $\d>0$ is small enough (here we note that~${h\bigl(1-\m'^2\bigr)}$ is smooth at $\m'=0$ although $h(\m)$ is not smooth at $\m=1$ ). Thus, the stationary phase theorem (see Lemma \ref{stphaselem2}\,(ii) with $\l=y$) implies
\smash{$
|I_2|\lesssim y^{-\frac{1}{2}}$}.
This proves \eqref{I_1stphase} and completes the proof of Proposition \ref{I_1prop}.
\end{proof}

\subsection[Improvement for 0<b<2]{Improvement for $\boldsymbol{0<b<2}$}\label{subsec0b2imp}

In this subsection, we improve Proposition \ref{I_1prop} near $\f=0$ when $0<b<2$.

\begin{Proposition}\label{I_1propimp}
We assume $0<b<2$.
Then there exist $\f_0>0$ such that
$
|I_{1,\s}(y,\f)|\lesssim
y^{(1-b)\n}$ for $y\gg 1$, $ \f\in [0,\f_0]$.
\end{Proposition}

\begin{Remark}
Combining with Proposition \ref{I_1prop}, we obtain the uniform estimates for $\f\in [0,\pi-\e]$ with arbitrary $\e>0$.
\end{Remark}

In order to prove it, we need a more information about the phase function $S_{1,\s}$. In the part~(i) of the next lemma, we use the assumption $0<b<2$ crucially. This is used to prove that the second term of the right-hand side in \eqref{discon} is harmless.

\begin{Lemma}\label{S_1implem}\quad
\begin{itemize}
\item[$(i)$]
There exists $\e_0>0$ such that
$
|\pa_mS_{1,\s}(m,y,\f)|\leq 2\pi-\e_0
$
for all $m\in \Omega_1$ and $y\geq 1$ and~${\f\in \bigl[0,(1-\frac{b}{2})\pi\bigr]}$.

\item[$(ii)$] There exists $\f_1,c_0>0$ such that
$
|\pa_mS_{1,\s}(m,y,\f)|\geq c_0
$
for $m\in \Omega_1$, $y\geq 1$ with $\m(m)\geq \frac{1}{2}y$ and $\f\in [0,\f_1]$.
\end{itemize}
\end{Lemma}

\begin{Remark}
If $b\geq 2$, then $\pa_mS_{1,\s}(m,y,\f)$ can take a value in $2\pi \mathbb{Z}\setminus \{0\}$. This prevents better estimates of $I_{1,\s}$.
\end{Remark}

\begin{proof}[Proof of Lemma~\ref{S_1implem}]
We recall $\pa_mS_{1,\s}(m,y,\f)=-b\s_1\cos^{-1}\bigl(\frac{\m(m)}{y}\bigr)+\s_2\f-\frac{\pi}{2}b$.

 $(i)$
This follows from a direct calculation: $\bigl|-b\s_1\cos^{-1}(z)+\s_2\f-\frac{\pi}{2}b\bigr|\leq \pi b+\f\leq (1+\frac{b}{2})\pi$ for $\f\in \bigl[0,\bigl(1-\frac{b}{2}\bigr)\pi\bigr]$ and $0\le z\le 1$. Then we set $\e_0=\bigl(1-\frac{b}{2}\bigr)\pi$, which is positive since $0<b<2$.

 $(ii)$ This also follows from a direct calculation
\begin{align*}
\left|-b\s_1\cos^{-1}(z)+\s_2\f-\frac{\pi}{2}b\right|\geq \left(\frac{\pi}{2}-\cos^{-1}z\right)b- |\f|\geq \frac{\pi}{6}b-|\f|
\end{align*}
for $\frac{1}{2}\leq z\leq 1$. Taking $\f_1=c_0=\frac{\pi b}{12}$, we obtain the bound in (ii).
\end{proof}

\begin{proof}[Proof of Proposition \ref{I_1propimp}]
Define $\f_0:=\min\bigl(\bigl(1-\frac{b}{2}\bigr)\pi,\f_1\bigr)>0$, where $\f_1$ is as in Lemma~\ref{S_1implem}.
Taking $\chi\in C^{\infty}(\R;[0,1])$ such that $\chi(\m)=1$ for $\m\leq 1/2$ and $\chi(\m)=0$ for $\m\geq 3/4$ and setting~${\overline{\chi}=1-\chi}$, we write
\begin{align*}
I_{1,\s}(y,\f)&= \sum_{m=1}^{\infty}(\chi(\m(m)/y)+\overline{\chi}(\m(m)/y)) \z_{1,\s}(m,y,\f){\rm e}^{{\rm i}S_{1,\s}(m,y,\f)} \\
&=: I_{1,1,\s}(y,\f)+I_{1,2,\s}(y,\f).
\end{align*}
By the support property of $\overline{\chi}$ and Lemma \ref{ampesti}\,(ii), we have
\begin{align*}
|\pa_{m}^{\a}(\overline{\chi}(\m(m)/y)\z_{1,\s}(m,y,\f))|\lesssim y^{(2-b)\n-\frac{1}{3}-\frac{\a}{3}}.
\end{align*}
Moreover, by virtue of Lemma \ref{phaseest}\,(i) and Lemma \ref{S_1implem}\,(i),~(ii), the phase function $S_{1,\s}$ satisfies the assumption of Proposition~\ref{discsum1} with $r=\min(c_0,\e_0)$, $M\sim y$ and $\rho=\frac{1}{3}$. Applying Proposition~\ref{discsum1} with $k=(2-b)\n-\frac{1}{3}$, for each $N>0$, we obtain $|I_{1,2,\s}(y,\f)|\lesssim y^{-N}$. Thus, it remains to estimate~${I_{1,1,\s}(y,\f)}$.

We write $I_{1,1,\s}$ as in \eqref{I1sigmaint} and consider the integral
\begin{align*}
I_{q}=\int_{\R}\z(m,y,\f){\rm e}^{{\rm i}S_{1,\s}(m,y,\f)+2\pi {\rm i}qm}{\rm d}m, \qquad \text{for}\quad \z\in \{\z_{1,\s}',\z_{1,\s}''\},
\end{align*}
where we set
\begin{align*}
\z'_{1,\s}(m,y,\f)=&\chi(\m(m)/y)\z_{1,\s}(m,y,\f),\\
\z''_{1,\s}(m,y,\f)=&\pa_m\z_{1,\s}'(m,y,\f)+{\rm i}(\pa_mS_{1,\s})(m,y,\f)\z_{1,\s}'(m,y,\f).
\end{align*}
As in the proof of Proposition \ref{I_1prop}, it suffices to prove the existence of $\f_0>0$ such that
\begin{align}\label{I_qest2}
|I_q|\lesssim \begin{cases}
y^{(1-b)\n}& \text{for}\  q=0,\\
y^{(1-b)\n}(1+|q|)^{-1}& \text{for}\  |q|\geq 1,
\end{cases}
\end{align}
and for $m\geq 1$, $y\gg 1$ and $\f\in [0,\f_0]$.
By Lemmas \ref{phaseest}\,(i), \ref{ampesti}\,(ii) and the support property of~$\chi$, for $\z\in\{\z_{1,\s}',\z_{1,\s}''\}$, we have
\begin{align}\label{zetapart2symbolest}
|\pa_m^{\a}\z(m,y,\f)|\lesssim y^{-b\n-\frac{1}{2}}m^{2\n-\a}(1+m\sin \f)^{-\n}
\end{align}
and $\supp \z(\cdot,y,\f)\subset \big\{m\geq 1\mid \m(m)\in \bigl[\m(1),\frac{3}{4}y\bigr]\big\}$.

First, we consider the case $|q|\geq 1$. In this case, we have
\begin{align*}
|\pa_m(S_{1,\s}(m,y,\f)+2\pi qm)|\gtrsim (1+|q|)
\end{align*}
by Lemma \ref{S_1implem}\,(i).
By using integration by parts many times, for each $N>0$, we have
\begin{align*}
|I_q|\lesssim y^{-N}(1+|q|)^{-N}
\end{align*}
due to Lemma \ref{phaseest}\,(i) and the estimate \eqref{zetapart2symbolest}.
This proves the second estimates of \eqref{I_qest2}.

Next, we consider the case $q=0$. By the change of variable $\m(m)(=bm+b\n)=y\m$,
\[
I_0=b^{-1}{\rm e}^{-{\rm i}\left(\s_2\frac{\f}{b}-\frac{\pi}{2}\right)b\n} y^{(2-b)\n+\frac{1}{2}}\int_{\R}{\rm e}^{{\rm i}yS(\m,\f) }\c_{y,\f}(\m) {\rm d}\m,
\]
where we set
\begin{align*}
S(\m,\f)=\s_1h_1(\m) +\left(\s_2\frac{\f}{b}-\frac{\pi}{2}\right)\m,\qquad \c_{y,\f}(\m)=y^{(b-2)\n+\frac{1}{2}}\z\left(\frac{y\m}{b}-\n,y,\f\right).
\end{align*}
By \eqref{zetapart2symbolest}, the support property of $\chi$ and $\f\lesssim\sin \f$, we have
\begin{align*}
|\pa_{\m}^{\a}\c_{y,\f}(\m)|\lesssim \m^{2\n-\a}(1+y \f\m)^{-\n},\qquad \supp \c_{y,\f}\subset \left[\m(1)y^{-1},\frac{3}{4}\right].
\end{align*}
Moreover, the phase function $S(\m,\f)$ satisfies the assumption of Proposition \ref{Stphasemovecrit} with the critical point $\m(\f)=\cos\left(\s_1\left(-\frac{\pi}{2}+\s_2\frac{\f}{b}\right)\right)$.
Now Proposition \ref{Stphasemovecrit} with $\l=y$ implies
\begin{align*}
|I_0|\lesssim y^{(2-b)\n+\frac{1}{2}}\cdot y^{-\n-\frac{1}{2}}=y^{(1-b)\n}.
\end{align*}
This completes the proof.
\end{proof}

\section{Proof of the main theorem}\label{proofsection}

\subsection{Proof of Theorem \ref{keythm}}
\label{subsec:proof}

The claim (i) for $b=1$ directly follows from \cite[equation~(4.45)]{BKO}. Hence we may assume $b\neq 1$.

Let $R_0\gg 1$. First, we prove the theorem for $|y|\leq R_0$. To do this, we follow the argument in \cite[Lemma 4.17]{BKO}.
Using the bound $\big|\tilde{I}_{\l}(w)\big|\leq {\rm e}^{|\re w|}\Gamma(\l+1)^{-1}$, $C_0^{\n}(t)=1$ and $\big|\n^{-1}C_{m}^{\n}(t)\big|\lesssim (1+m)^{2\n-1}$ for $m\geq 1$ (see \cite[equation~(4.16)]{BKO}, \cite[Fact 4.8]{BKO} and \eqref{Gegenunif}), we have
\begin{align*}
|\sI(b,\nu;-{\rm i}y;\cos\f)|\lesssim{}& \sum_{m=0}^{\infty}(1+m)\cdot |y|^{bm}\cdot \frac{1}{\Gamma(bm+b\n+1)}\cdot (1+m)^{2\n-1}\\
={}&\sum_{m=0}^{\infty}\frac{(1+m)^{2\n}}{\Gamma(bm+b\n+1)}\cdot |y|^{bm},
\end{align*}
which is convergent and bounded by a constant independent of $y\in \R$ with $|y|\leq R_0$.

Now the estimate for $y\geq R_0$ follows from the decomposition \eqref{sIdecom} of $\sI$ and Propositions~\ref{Ryphiprop}, \ref{I_2prop}, \ref{I_3prop}, \ref{I_1prop}, \ref{I_1propimp}.

The case $y\leq -R_0$ is similarly dealt with due to the formula
\begin{align*}
\sI(b,\nu;{\rm i}y;\cos\f)=&L_{b,\n}{\rm e}^{{\rm i}b\n\pi}y^{-b\n}\sum_{m=0}^{\infty}{\rm e}^{\frac{\pi}{2}bm{\rm i}}J_{b(m+\n)}(y)(m+\nu)\n^{-1}C^{\nu}_m(\cos\f),
\end{align*}
which in turn follows from the identity $J_{\m}\bigl({\rm e}^{{\rm i}\pi}y\bigr)={\rm e}^{{\rm i}\m\pi}J_{\m}(y)$ and \eqref{sumbesell}.

\subsection[Proof of Theorem 1.1]{Proof of Theorem \ref{mainthmSt}}\label{subsec:Stpf}

Now we use the following general result due to Keel--Tao.

\begin{Theorem}[\cite{KT}]\label{KeelTao}
Let $X$ be a measure space and $\{U(t)\}_{t\in \R}$ be a bounded family of continuous linear operators on $L^2(X)$ such that there are $C>0$ and $\s>0$ such that
\[
\|U(t)U(s)^*\|_{L^1(X)\to L^{\infty}(X)}\leq C|t-s|^{-\s}
\]
 for all $t,s\in \R$ with $t\neq s$. Let $(p,q)\in [2,\infty]^2$ such that
\begin{align}\label{addmissible}
\frac{1}{p}+\frac{\s}{q}=\frac{\s}{2} \qquad \text{and}\qquad (p,q,\s)\neq (2,\infty,1).
\end{align}
Then there exists $C>0$ such that
\begin{gather*}
\|U(t)u\|_{L^p(\R; L^q(X))}\leq C\|u\|_{L^2(X)},\\
 \left\|\int_{-\infty}^tU(t)U(s)^*f(s){\rm d}s \right\|_{L^{p_1}(\R; L^{q_1}(X))} \leq C\|f\|_{L^{p_2^*}(\R; L^{q_2^*}(X))},
\end{gather*}
where $(p_1,q_1)$, $(p_2,q_2)$ satisfy \eqref{addmissible} and $r^*$ denotes the H\"older conjugate of $r$: $r^*=r/(r-1)$.

\end{Theorem}

First we note that the operator norm of an operator from $L^1$ to $L^{\infty}$ is equal to the $L^{\infty}$-norm of its integral kernel.
We may assume $T\leq \frac{\pi}{4}$ due to the argument in Appendix \ref{finitetime}.

First, we consider the cases
\begin{itemize}\itemsep=0pt
\item $n=1$ and $a\ge 2-4k$,
\item $n\ge 2$ and $(0<a\leq 1$ or $a=2)$.
\end{itemize}
We claim
\begin{align}\label{Hkadispersive}
\big|{\rm e}^{-{\rm i}tH_{k,a}}(x,x')\big|\lesssim |t|^{-\s_{k,a}}\qquad \text{for}\quad |t|\leq \frac{\pi}{2}.
\end{align}
Let us prove the claim for the case $n=1$ and $a\ge 2-4k$.
In this case, we have $\s_{k,a}\ge\frac{1}{2}$ and~\cite[Proposition 4.29]{BKO} implies
\begin{align}
{\rm e}^{-{\rm i}tH_{k,a}}(x,x')={}&\Gamma(\sigma_{k,a})\frac{{\rm e}^{{\rm i}\frac{|x|^a+|x'|^a}{a}\cot(t)}}{({\rm i}\sin(t))^{\s_{k,a}}}\nonumber\\
&\times\left(\tilde{I}_{\s_{k,a}-1}\left(\frac{2|xx'|^{\frac{2}{a}}}{a{\rm i}\sin(t)}\right)+\frac{xx'}{(a{\rm i}\sin(t))^{\frac{2}{a}}}\tilde{I}_{\s_{k,a}-1+\frac{2}{a}}\left(\frac{2|xx'|^{\frac{2}{a}}}{a{\rm i}\sin(t)}\right)\right).\label{1dim}
\end{align}
From the asymptotic expansion of the I-Bessel function \cite[Section~7.23\,(2), (3)]{W}, we see
\begin{align}\label{1dim2}
I_{\n}(\pm {\rm i}y)=\frac{{\rm e}^{\pm {\rm i}y}}{(\pm2\pi {\rm i}y)^{1/2}}\bigl(1+O\bigl(y^{-1}\bigr)\bigr)
\end{align}
as $y\to\infty$.
Hence \eqref{1dim2} and $\tilde{I}_{\n}(z)=(\frac{z}{2})^{-\n}I_{\n}(z)$ imply that there exists $C>0$ such that \smash{$\big|\tilde{I}_{\s_{k,a}-1}({\rm i}y)\big|+\big|y^{\frac{2}{a}}\tilde{I}_{\s_{k,a}-1+\frac{2}{a}}({\rm i}y)\big|\le C(1+|y|)^{-\s_{k,a}+\frac{1}{2}}\le C$}
for any $y\in\R$.
Therefore, our claim \eqref{Hkadispersive} follows from \eqref{1dim}.

Let us consider the case $n\ge 2$ and $(0<a\leq 1$ or $a=2)$.
Since ${\rm d}\m^k_{\hat{x}}$ is a probability measure, we see
\begin{align*}
\left\|\int_{\R^n}f\bigl(\x\cdot \hat{x'}\bigr) {\rm d}\m_{\hat{x}}^k(\x)\right\|_{L^{\infty}(\R^n\times \R^n)}\leq \|f\|_{L^{\infty}([-1,1])}.
\end{align*}
Using this estimate and \eqref{intkergen}, we obtain
\begin{align*}
\big|{\rm e}^{-{\rm i}tH_{k,a}}(x,x')\big|\leq& c_{k,a}\frac{1}{|\sin (t)|^{\s_{k,a}}} \sup_{\y\in [-1,1]}\bigg|\sI\bigg(\frac{2}{a},\frac{a}{2}(\s_{k,a}-1); -{\rm i}\frac{2|x|^{\frac{a}{2}}|x'|^{\frac{a}{2}} }{a\sin(t)}; \y\bigg)\bigg|.
\end{align*}
Therefore, the claim \eqref{Hkadispersive} follows from Theorem \ref{keythm}\,(ii).

By the claim and Theorem \ref{KeelTao} with $U_{\pm}(t):=1_{[0,T]}(t){\rm e}^{\mp {\rm i}tH_{k,a}}$, we obtain the homogeneous Strichartz estimates \eqref{homStr} and
\begin{align*}
&\left\|\int_{0}^t{\rm e}^{-{\rm i}(t-s)H_{k,a}}f(s){\rm d}s \right\|_{L^{p_1}([0,T]; L^{q_1})} \leq C\|f\|_{L^{p_2^*}([0,T]; L^{q_2^*})},\\
&\left\|\int_{0}^t{\rm e}^{{\rm i}(t-s)H_{k,a}}f(-s){\rm d}s \right\|_{L^{p_1}([0,T]; L^{q_1})} \leq C\|f\|_{L^{p_2^*}([-T,0]; L^{q_2^*})}.
\end{align*}
Since
\[
\left\|\int_{0}^t{\rm e}^{{\rm i}(t-s)H_{k,a}}f(-s){\rm d}s \right\|_{L^{p_1}([0,T]; L^{q_1})}=\left\|\int_{0}^t{\rm e}^{-{\rm i}(t-s)H_{k,a}}f(s){\rm d}s \right\|_{L^{p_1}([-T,0]; L^{q_1})},
\] we also obtain the inhomogeneous Strichartz estimates \eqref{inhomStr}.

Next, we consider the case when $1<a<2$ and $k\equiv 0$. Let $0<\f_0<2\pi$.
Take a finite partition of unity $\{\chi_j\}_{j=1}^N\subset C^{\infty}\bigl(\mathbb{S}^{n-1};[0,1]\bigr)$ on the sphere such that
$
\omega,\y\in \supp \chi_j\Rightarrow \omega\cdot \y\geq \cos\f_0$.
Due to \eqref{intker0} and Theorem \ref{keythm}\,(i), there exists $C>0$ such that $\big|\chi_j(\hat{x}){\rm e}^{- {\rm i}tH_{k,a}}(x,x')\chi_j\bigl(\hat{x'}\bigr)\big|\leq C|t|^{-\s_{0,a}}$ for $|t|\leq \frac{\pi}{2}$ and $x,x'\in \R^n\setminus \{0\}$, where we recall $\hat{x}=x/|x|$. Hence,
\begin{align*}
\big\|\chi_j(\hat{x}){\rm e}^{- {\rm i}(t-s)H_{k,a}}(x,x')\chi_j\bigl(\hat{x'}\bigr)\big\|_{L^{\infty}(\R^n_x\times \R^n_{x'})}\leq C|t-s|^{-\s_{0,a}}
\end{align*}
for $|t|, |s|\leq \frac{\pi}{4}$ with $t\neq s$.
Applying Theorem \ref{KeelTao} with $U(t):=1_{[-T,T]}(t)\chi_j(\hat{x}){\rm e}^{- {\rm i}tH_{k,a}}$, we obtain
\smash{$
\big\|\chi_j(\hat{x}) {\rm e}^{-{\rm i}tH_{k,a}}u\big\|_{L^p([-T,T]; L^q)}\leq C\|u\|_{L^2}$}.
Since the number $N$ is finite and since $\sum_{j=1}^N\chi_j=1$, we obtain the homogeneous Strichartz estimates \eqref{homStr}.

The inhomogeneous Strichartz estimates \eqref{inhomStr} when $1<a<2$ and $k\equiv 0$ follow from the Christ--Kiselev lemma \cite{CK} and a complex interpolation since the end-point case $(p_j,q_j)=(2,\frac{2\s_{0,a}}{\s_{0,a}-1})$ is excluded (see the proof of \cite[Theorem 6]{BT}).

\subsection[Proof of Theorem 1.3]{Proof of Theorem \ref{coroSt}}
Let us define elements of $\sl(2,\R)$ as follows
$
\bfe^+:=\left(\begin{smallmatrix}0&1\\0&0\end{smallmatrix}\right)$, $
\bfh:=\left(\begin{smallmatrix}1&0\\0&-1\end{smallmatrix}\right)$, $
\bfe^-:=\left(\begin{smallmatrix}0&0\\1&0\end{smallmatrix}\right)$.
Assume~${\sigma_{k,a}>0}$.
Then it is proved in \cite[Theorem 3.30]{BKO} that there exists a unitary representation~$\Omega_{k,a}$ of the universal cover \smash{$\widetilde{\SL}(2,\R)$} of $\SL(2,\R)$ on the Hilbert space $L^2(\R^n,\vartheta_{k,a}(x){\rm d}x)$ (see~\eqref{vartheta} for the definition of $\vartheta_{k,a}$) satisfying
\begin{alignat*}{3}
& \Omega_{k,a}\bigl(\exp\bigl(t\bfe^+\bigr)\bigr)u(x)={\rm e}^{\frac{{\rm i}t|x|^a}{a}}u(x),\qquad&& \Omega_{k,a}(\exp(t\mathbf{h}))u(x)={\rm e}^{t\sigma_{k,a}}u\bigl({\rm e}^{\frac{2t}{a}}x\bigr),& \\
& \Omega_{k,a}(\exp(t\bfe^-))u(x)={\rm e}^{\frac{{\rm i}t|x|^{2-a}\Delta_k}{a}}u(x),\qquad&&\Omega_{k,a}(\exp(t(\bfe^--\bfe^+)))u(x)={\rm e}^{-{\rm i}tH_{k,a}}u(x)&
\end{alignat*}
for any $u\in L^2(\R^n,\vartheta_{k,a}(x){\rm d}x)$.
Let $\theta=\arctan(t)$. We apply $\Omega_{k,a}$ to the identity
\[
\exp\bigl(\theta\bigl(\bfe^--\bfe^+\bigr)\bigr)=\exp\bigl(-t\bfe^+\bigr)\exp\left(\frac{\log\bigl(1+t^2\bigr)}{2}\bfh\right)\exp(t\bfe^-)
\]
in \smash{$\widetilde{\SL}(2,\R)$}, and we obtain
\begin{align}\label{SL2}
{\rm e}^{-{\rm i}\theta H_{k,a}}u(x)
={\rm e}^{\frac{-{\rm i}t|x|^a}{a}}\bigl(1+t^2\bigr)^{\frac{\sigma_{k,a}}{2}}\left({\rm e}^{\frac{{\rm i}t|x|^{2-a}\Delta_k}{a}}u\right)\bigl(\bigl(1+t^2\bigr)^{1/a}x\bigr)
\end{align}
for any $u\in L^2(\R^n,\vartheta_{k,a}(x){\rm d}x)$.

\begin{proof}[Proof of Theorem \ref{coroSt}]
From \eqref{SL2}, the integral kernel \smash{${\rm e}^{\frac{{\rm i}t|x|^{a-2}\Delta_k}{a}}(x,y)$} of \smash{${\rm e}^{\frac{{\rm i}t|x|^{a-2}\Delta_k}{a}}$} equals
\begin{align*}
{\rm e}^{\frac{{\rm i}t|x|^a}{(1+t^2)a}}\bigl(1+t^2\bigr)^{\frac{-\sigma_{k,a}}{2}}{\rm e}^{-{\rm i}\arctan(t)H_{k,a}}\bigl(\bigl(1+t^2\bigr)^{\frac{-1}{a}}x,y\bigr).
\end{align*}
Therefore, the similar arguments as the proof of Theorem \ref{mainthmSt} and the equation
\[
\bigl(1+t^2\bigr)^{\frac{1}{2}}\sin(\arctan(t))=t
\]
 imply
\smash{$\big|{\rm e}^{\frac{{\rm i}t|x|^{a-2}\Delta_k}{a}}(x,y)\big|\lesssim |t|^{-\sigma_{k,a}}$} for $t\in\R$, and Theorem \ref{coroSt} follows.
\end{proof}

\appendix

\section{Asymptotic behavior of special functions}

\subsection{Leibniz's rule}

We frequently use the following formula, which directly follows from Leibniz's rule.

\begin{Lemma}\label{Leibnizrule}
Let $r$ be a smooth function and $N\in\mathbb{N}\setminus\{0\}$. Then there exist smooth functions~$b_{jN}$ such that
\begin{gather*}
\bigl(\pa_x\circ r(x)^{-1}\bigr)^N=\frac{1}{r(x)^N}\!\sum_{j=0}^Nb_{jN}(x)\pa_{x}^j\qquad \text{with}\quad |b_{jN}(x)|\lesssim \sum_{k=1}^{N-j}\sum_{\substack{\ell_1+\cdots+\ell_k=N-j,\\ \ell_1,\hdots,\ell_k\geq 1}}\prod_{i=1}^k\!\frac{\big|r^{(\ell_i)}(x)\big|}{|r(x)|}.
\end{gather*}
\end{Lemma}

\begin{Remark}
More precisely, we have
\begin{align*}
\bigl(\pa_x\circ r(x)^{-1}\bigr)^N=\frac{1}{r(x)^N}\pa_x^N+\frac{1}{r(x)^N}\sum_{j=0}^{N-1}\Bigg(\sum_{k=1}^{N-j}\sum_{\substack{\ell_1+\cdots+\ell_k=N-j,\\ \ell_1,\hdots,\ell_k\geq 1}}C_{\ell_1\ell_2\hdots \ell_k}^{jN}\prod_{i=1}^{k}\frac{r^{(\ell_i)}(x)}{r(x)} \Bigg)\pa_{x}^j
\end{align*}
with constants \smash{$C_{\ell_1\ell_2\hdots \ell_m}^{jN}>0$}.
\end{Remark}

\subsection{Non-stationary phase theorem with a singular amplitude}

The following lemma is more or less well known and an easy consequence of integration by parts. We give a proof for completeness of the paper.

\begin{Lemma}\label{singularasym}
Let $\m>-1$, $\c\in C^{\infty}((0,\infty))$ which is supported in $x\leq 10$ such that for $\a\in\mathbb{N}$, we have $|\pa_x^{\a}\c(x)|\lesssim |x|^{\m-\a}$ for $x\in (0,\infty)$. Let $f\in C^{\infty}(\R)$ satisfy $\im f(x)\geq 0$, $f'(x)\neq 0$ for~${x\in \supp \c}$ and $f(0)\in \R$. We define
\begin{align*}
b_{\m}(\l):={\rm e}^{-{\rm i}\l f(0)}\int_{0}^{\infty}\c(x){\rm e}^{{\rm i}\l f(x)}{\rm d}x\qquad \text{for}\quad \l\gtrsim 1.
\end{align*}
Then for each $\a\in\mathbb{N}$
\begin{align}\label{singbmues}
|\partial_{\l}^{\a}b_{\m}(\l)|\lesssim \l^{-\m-1-\a},\qquad \l\gtrsim 1.
\end{align}
In particular, $\big|\int_{0}^{\infty}\c(x){\rm e}^{{\rm i}\l f(x)}{\rm d}x\big|\lesssim \l^{-\m-1}$ for $\l\gtrsim 1$.
\end{Lemma}

\begin{proof}
We only need to prove these estimates for sufficiently large $\l$.

First, we prove \eqref{singbmues} for $\a=0$. Let $\chi\in C_c^{\infty}(\R)$ such that $\chi(x)=1$ for $|x|\leq 1/2$ and~${\chi(x)=0}$ for $|x|\geq 1$. Define $\chi_{\l}(x)=\chi(\l x)$ and $\overline{\chi_\l}(x)=1-\chi_{\l}(x)$. Then we have
\begin{align}\label{singbmues1}
\left|\int_{0}^{\infty}\c(x)\chi_{\l}(x){\rm e}^{{\rm i}\l f(x)}{\rm d}x\right|\lesssim \int_{0}^{1/\l}x^{\m}{\rm d}x\lesssim {\l}^{-\m-1}.
\end{align}
On the other hand, the integration by parts yields
\begin{align*}
\left|\int_{0}^{\infty}\c(x)\overline{\chi_{\l}}(x){\rm e}^{{\rm i}\l f(x)}{\rm d}x\right|=\l^{-N}\left|\int_{0}^{\infty}L^N(\c(x)\overline{\chi_{\l}}(x)){\rm e}^{{\rm i}\l f(x)}{\rm d}x\right|,
\end{align*}
where we define $L=\pa_x\circ ({\rm i}f'(x))^{-1}$. By using Leibniz's rule, Lemma \ref{Leibnizrule} and the assumption $f'\neq 0$ on $\supp \c$, we obtain $|L^{N}(\c(x)\overline{\chi_\l}(x))|\lesssim |x|^{\m-N}$. Hence, for $N>\m+1$,
\begin{align}\label{singbmues2}
\left|\int_{0}^{\infty}\c(x)\overline{\chi_{\l}}(x){\rm e}^{{\rm i}\l f(x)}{\rm d}x\right|\lesssim \l^{-N}\int_{\frac{1}{2\l}}^{10}x^{\m-N}{\rm d}x \lesssim \l^{-\m-1}.
\end{align}
The inequalities \eqref{singbmues1} and \eqref{singbmues2} imply \eqref{singbmues} for $\a=0$.

To deal with the case $\a\geq 1$, we write
\begin{align*}
\pa_{\l}^{\a}b_{\m}(\l)={\rm i}^{\a}{\rm e}^{-{\rm i}\l f(0)}\int_0^{\infty}(f(x)-f(0))^{\a}\c(x){\rm e}^{{\rm i}\l f(x)}{\rm d}x.
\end{align*}
By Taylor's theorem, we have $\pa_x^{N}(\c(x)(f(x)-f(0))^{\a})=O\bigl(|x|^{\m+\a-N}\bigr)$ for $x\in \supp \c$. Moreover, $\big|{\rm e}^{-{\rm i}\l f(0)}\big|=1$ due to the assumption $f(0)\in \R$.
Hence the same argument as the case $\a=0$ gives \eqref{singbmues} for $\a\geq 1$.
\end{proof}

\subsection{Stationary phase theorem}

In this paper, we use the following versions of the stationary phase theorem (the van der Corput lemma). The proof is same as the proof of \cite[Lemma 1.1.2]{So} (see also in \cite[p.~334]{S}). The statement about uniformity in a parameter follows from its proof and Lemma \ref{Leibnizrule}. We consider an integral
\begin{align*}
I(\l):=\int_{\R}\c(x){\rm e}^{{\rm i}\l f(x)}{\rm d}x\qquad \text{for}\quad \l\gtrsim 1.
\end{align*}

\begin{Lemma}\label{stphaselem}
Let $k\geq 2$ be an integer and $j\geq 0$, $x_0\in \R$, $\c\in C^{\infty}(\R\setminus \{x_0\})\cap C_c(\R)$ and~${f\in C^{\infty}(\R;\R)}$ such that $|\pa_x^{\a}\c(x)|\lesssim |x-x_0|^{j-\a}$, $|f(x)|\lesssim |x-x_0|^k$ and $|f'(x)|\gtrsim |x-x_0|^{k-1}$ for $x\in \supp \c\setminus \{x_0\}$.
Then \smash{$|I(\l)|\lesssim \l^{-\frac{j+1}{k}}$} for $\l\gtrsim 1$.
\end{Lemma}

As a result, we also obtain the following version.

\begin{Lemma}\label{stphaselem2}
Let $\c\in C^{\infty}(\R\setminus \{0\})\cap C_c(\R)$ and $f\in C^{\infty}(\R;\R)$.
\begin{itemize}\itemsep=0pt
\item[$(i)$] Let $\n\geq 0$. Suppose that $|\pa_x^{\b}\c(x)|\lesssim |x|^{\n-\b}$ for $x\in \R\setminus\{0\}$ and $f''(x)\neq 0$ for $x\in \supp \c$. Then \smash{$|I(\l)|\lesssim \l^{-\frac{1}{2}}$} for $\l\gtrsim 1$.

\item[$(ii)$] Suppose that \smash{$\big|\pa_x^{\b}\c(x)\big|\lesssim |x|^{\frac{1}{2}-\b}$} for $x\in \R\setminus\{0\}$, $f'(0)=0$ and $f'''(0)\neq 0$. If $\supp \c$ is sufficiently close to $0$, then \smash{$|I(\l)|\lesssim \l^{-\frac{1}{2}}$} for $\l\gtrsim 1$.
\end{itemize}
\end{Lemma}

\begin{Remark}\quad
\begin{itemize}\itemsep=0pt
\item[(1)] The main difference between Lemma \ref{stphaselem} and this lemma is that we do not assume $f'(0)=0$ in (i) for example.
\item[(2)] The estimates in (i) and (ii) are optimal. If in addition we assume that $\supp \c$ is close to $0$ for the case (i), then the optimal estimate would be $|I(\l)|\lesssim \l^{-\frac{1}{2}(1+\n)}$.
 \end{itemize}
\end{Remark}

\begin{proof}[Proof of Lemma~\ref{stphaselem2}]
The proof of the part (i) is almost same as the case (ii). Thus we only deal with the case (ii).

$(ii)$ We may assume $f(0)=0$ since $\big|{\rm e}^{{\rm i}\l f(0)}\big|=1$. Moreover, we may assume that $a:=f''(0)$ is sufficiently small since our result directly follows from \cite[p.334, Corollary]{S} for $a$ away from~$0$. For simplicity, we also assume $a>0$ and $f'''(0)\geq 8$.

First, we consider the case where $a^3\l\ll 1$. We write $f'(x)=ax+\frac{f'''(0)}{2}x^2+g(x)$ with ${g(x)=O\bigl(|x|^3\bigr)}$. Then we can take $\supp \c$ small enough such that $|f'(x)|\gtrsim |x|^2$ for $x\in \supp \c$ with \smash{$|x|\geq \l^{-\frac{1}{3}}$} due to the assumptions $a^3\l\ll 1$ and $f'''(0)\neq 0$.
Take $\chi\in C_c^{\infty}((-2,2);[0,1])$ such that $\chi(x)=1$ for $|x|\leq 1$ and set \smash{$\chi_{\l}(x)=\chi\bigl(\l^{\frac{1}{3}}xv)$} and $\overline{\chi_{\l}}=1-\chi_{\l}$. Then
\[
\left|\int_{\R}\c(x)\chi_{\l}(x){\rm e}^{{\rm i}\l f(x)}{\rm d}x\right|\lesssim \int_{|x|\leq 2\l^{-\frac{1}{3}}}|x|^{\frac{1}{2}} {\rm d}x\lesssim \l^{-\frac{1}{3}(1+\frac{1}{2})}=\l^{-\frac{1}{2}}.
\]
 On the other hand, the integration by parts yields
\begin{align*}
\left|\int_{\R}\c(x)\overline{\chi_{\l}}(x){\rm e}^{{\rm i}\l f(x)}{\rm d}x\right|&\leq \l^{-1}\int_{\R}|\pa_x\left(f'(x)^{-1}\c(x)\overline{\chi_{\l}}(x) \right)|{\rm d}x\\
&\lesssim \l^{-1}\int_{|x|\geq \l^{-\frac{1}{3}}}|x|^{-\frac{5}{2}}{\rm d}x\lesssim \l^{-\frac{1}{2}}.
\end{align*}
Thus we obtain $|I(\l)|\lesssim \l^{-\frac{1}{2}}$.

Next we consider the case where $a^3\l\gtrsim 1$. Taking $\supp \c$ sufficiently close to $0$, we may assume the following three conditions hold: Critical points of $f$ (on $\supp \c$) are $x=0$ or $x=x_0(a)$ with ${|x_0(a)|\in \bigl[\frac{a}{4}, \frac{3}{4}a\bigr]}$ \big(due to the factorization \smash{$f'(x)=x\bigl(a+\frac{f'''(0)}{2}x+h(x)\bigr)$} with $h(x)=O\bigl(x^2\bigr)$\big). The second condition is $|f''(x_0(a))|\gtrsim |a|$. The third condition is $|f'(x)|\gtrsim |ax|$ for $|x|\geq a$.

By scaling, we have
\begin{align}\label{stphase2scaling}
I(\l)=a^{\frac{3}{2}}\int_{\R}\c_a(x){\rm e}^{{\rm i}\l a^3f_a(x)}{\rm d}x,\qquad f_a(x):= a^{-3}f(ax),\qquad \c_a(x):=a^{-\frac{1}{2}}\c(ax).
\end{align}
We observe \smash{$f_a(x)=\frac{1}{2}x^2+\frac{f'''(0)}{6}x^3+O\bigl(a|x|^4\bigr)$} and \smash{$\big|\pa_x^{\b}\c_a(x)\big|\lesssim |x|^{\frac{1}{2}-\b}$}. Moreover, critical points of~$f_a$ are $x=0$ or $x=a^{-1}x_0(a)$ with $\big|a^{-1}x_0(a)\big|\in \bigl[\frac{1}{4},\frac{3}{4}\bigr]$, $|f_a''(x)|\gtrsim 1$ and $|f_a'(x)|\gtrsim |x|$ for $|x|\geq 1$. Then we divide the integral \eqref{stphase2scaling} into three parts $I_1(\l)$, $I_2(\l)$, $I_3(\l)$, where the integrand of~$I_j(\l)$ is supported close to $0$ for $j=1$, $a^{-1}x_0(a)^{-1}$ for $j=2$ and $|f_a''(x)|\gtrsim 1$ on the supports of the integrands of $I_1(\l)$ and $I_2(\l)$ for sufficiently small $a$. Then Lemma \ref{stphaselem} with the large parameter~$\l a^3$ gives \smash{$|I_1(\l)|\lesssim a^{\frac{3}{2}}\bigl(\l a^3\bigr)^{-\frac{1}{2}(1+\frac{1}{2})}\lesssim \l^{-\frac{1}{2}}$} and \smash{$|I_2(\l)|\lesssim a^{\frac{3}{2}}\bigl(\l a^3\bigr)^{-\frac{1}{2}}= \l^{-\frac{1}{2}}$} (for the latter case, we use $f'(x(a))=0$ and use this lemma for the phase function ${f(x)-f(x(a))}$). Moreover, integrating by parts many times, we have \smash{$|I_3(\l)|\lesssim a^{\frac{3}{2}}\bigl(\l a^3\bigr)^{-N}$} for all $N>0$. Taking~${N=\frac{1}{2}}$, we obtain \smash{$|I_3(\l)|\lesssim \l^{-\frac{1}{2}}$} and hence \smash{$|I(\l)|\lesssim \l^{-\frac{1}{2}}$}.
\end{proof}

\subsection{Asymptotics of the Bessel function}\label{AppBessel}

In this appendix, we prove Proposition \ref{Besselasymp}.
First, we seek a nice integral representation approximating the Bessel function $J_{\m}(y)$.

\begin{Lemma}\label{appBeselreduce}
There exists $\chi_1\in C_c^{\infty}\bigl(\bigl(-\frac{3\pi}{4},\frac{3\pi}{4}\bigr);[0,1]\bigr)$ such that $\chi_1(w)=1$ for $|w|\leq \frac{2\pi}{3}$ and~${\chi_1(w)=\chi_1(-w)}$ such that
\begin{align}\label{Besselgoodrep}
J_{\m}(y)=\frac{1}{2\pi}\int_{\R}{\rm e}^{{\rm i}(y\sin w-\m w)}\chi_1(w){\rm d}w+R_0(\m,y)\qquad \text{for}\quad \m\geq 0,\qquad y\geq 1,
\end{align}
where $R_0(\m,y)$ satisfies $|\pa_{\m}^{\a}R_0(\m,y)|\leq C_{N\a}(1+y+\m)^{-N}$ for each $N\in\mathbb{N}$.
\end{Lemma}

\begin{proof}
It is known (see \cite[Section~6.21\,(3)]{W}, see also \cite[equation~(9.7)]{IJ}) that the Hankel function can be written as
\begin{align}\label{Hankelint}
H_{\m}(y)=\frac{1}{\pi {\rm i}}\left(\int_{-\infty}^0+\int_0^{\pi {\rm i}}+\int_{\pi {\rm i}}^{\infty+\pi {\rm i}}\right){\rm e}^{y\sinh w-\m w}{\rm d}w
\end{align}
and $\re H_{\m}(y)=J_{\m}(y)$. Since the curve $[-\infty,0]\cup [0,\pi {\rm i}]\cup [\pi {\rm i},\infty+\pi {\rm i}]$ is not smooth, we will deform it to remove the corner at $\pi {\rm i}$ by using the Cauchy integral formula, where we note that the corner at $0$ is not involved with the asymptotics of $J_{\n}(y)$.
More precisely, we take a~curve~${\c\colon\R\to \mathbb{C}}$ such that $\c$ is smooth and $|\c'(r)|=1$ except at $0$, and
\begin{gather*}
\c(r)=\begin{cases}
r& \text{for}\quad r\leq 0,\\
 {\rm i}r& \text{for}\quad 0\leq r\leq \frac{3}{4}\pi,\\
r-\frac{2}{3}\pi-L+ {\rm i}\pi& \text{for}\quad r\geq \frac{3}{4}\pi+L,
\end{cases} \\
\begin{cases}\re \c(r)\geq 0,\\ \im \c(r)\in \bigl[\frac{2\pi}{3},\pi\bigr], \end{cases}\qquad \text{for}\quad \frac{2}{3}\pi\leq r\leq \frac{3}{4}\pi+2L
\end{gather*}
with a constant $L>0$. Since $\re (y\sinh w-\m w)=y\sinh(\re w)\cos(\im w)-\m \re w$ and $y,\m\geq 0$, we have
\begin{align}\label{Besphaseleq0}
\re (y\sinh w-\m w)|_{w=\c(r)}\leq 0\qquad \text{for}\quad r\geq 0.
\end{align}

Now we take $\chi_1\in C_c^{\infty}\bigl(\bigl(-\frac{3\pi}{4},\frac{3\pi}{4}\bigr);[0,1]\bigr)$, $\chi_2\in C_c^{\infty}\bigl(\bigl[\frac{2\pi}{3},\frac{3\pi}{4}+2L\bigr];[0,1]\bigr)$ and $\chi_3\in C^{\infty}(\R;[0,1])$ such that $\chi_1+\chi_2+\chi_3=1$ on $[0,\infty)$ and
\begin{gather*}
\chi_1(r)=\chi_1(-r),\qquad \chi_1(r)=1\qquad \text{for}\quad |r|\leq \frac{2\pi}{3},\\
 \chi_3(r)=\begin{cases}0& \text{for}\quad r\leq \frac{3}{4}\pi+L,\\ 1& \text{for}\quad r\geq \frac{3}{4}\pi+2L. \end{cases}
\end{gather*}

By the Cauchy's integral formula, we rewrite \eqref{Hankelint} as
\begin{gather*}
H_{\m}(y)=\frac{1}{\pi {\rm i}}\int_{-\infty}^0{\rm e}^{y\sinh w-\m w}{\rm d}w+\frac{1}{\pi}\int_0^{\pi}{\rm e}^{{\rm i}(y\sin w-\m w)}\chi_1(w){\rm d}w+\tilde{R}_0(\m,y),\\
\tilde{R}_0(\m,y)=\frac{1}{\pi {\rm i}}\int_{\R}{\rm e}^{y\sinh \c(r)-\m \c(r)}\chi_2(r)\tilde{\c}'(r){\rm d}r\\
\phantom{\tilde{R}_0(\m,y)=}{}+\frac{{\rm e}^{-{\rm i}\pi\m}}{\pi {\rm i}}\int_{\R}{\rm e}^{-y\sinh w-\m w}\chi_3\left(w+\frac{2}{3}\pi+L\right){\rm d}w.
\end{gather*}
Since the first term is purely imaginary, we obtain
\begin{align*}
J_{\m}(y)=\re H_{\m}(y)=\frac{1}{2\pi}\int_{\R}{\rm e}^{{\rm i}(y\sin w-\m w)}\chi_1(w){\rm d}w+\re \tilde{R}_0(\m,y),
\end{align*}
where we use $\chi_1(r)=\chi_1(-r)$, \smash{$\supp \chi_1\subset \bigl(-\frac{3}{4}\pi,\frac{3}{4}\pi\bigr)$} and $y\sin(- w)-\m(-w)=-(y\sin w-\m w)$. Now we set $R_0(\m,y):=\re \tilde{R}_0(\m,y)$.
It suffices to prove $|\pa_{\m}^{\a}\tilde{R}_0(\m,y)|\leq C_{N\a}(1+y+\m)^{-N}$ for each $N\in\mathbb{N}$, $\m\geq 0$ and $y\geq 1$.

We note that $|\pa_r(y\sinh\c(r)-\m\c(r))|\geq |y\cosh (\re \c(r))\cos(\im \c(r))-\m|\gtrsim y+\m$ for $r\in \supp \chi_2$, where we use $|\c'(r)|=1$ and $\re\c\geq 0$, $\im \c\in [\frac{2}{3}\pi,\pi]$ on $\supp \chi_2$. Moreover, $|\pa_w(-y\sinh w-\m w)|=y\cosh w+\m \gtrsim y+\m$. Thus, the integration by parts with \eqref{Besphaseleq0} yields~${\big|\pa_{\m}^{\a}\tilde{R}_0(\m,y)\big|\leq C_{N\a}(1+y+\m)^{-N}}$ for each $N\in\mathbb{N}$, $\m\geq 0$ and $y\geq 1$.
\end{proof}

\begin{Remark}
This lemma might be proved also by integration by parts and by Schl\"afli's formula
\begin{align*}
J_{\m}(y)=\frac{1}{\pi}\int_0^{\pi}\cos (y(\sin w)-\m w){\rm d}w-\frac{\sin (\pi\m)}{\pi}\int_0^{\infty}{\rm e}^{-y(\sinh w)-\m w}{\rm d}w,
\end{align*}
which follows from \eqref{Hankelint}.
However, to do this, we need to calculate the boundary terms at~${w=\pi}$ in the first term and the one at $w=0$ in the second term explicitly. In the above proof, we avoid it and use the deformation of the integral instead.
\end{Remark}

Now we study the first term (we call it $F_1(y,\m)$) of the right-hand side in \eqref{Besselgoodrep}. To do this, we use a variant of the stationary phase theorem. Roughly speaking, the phase function~${y\sin w-\m w}$ of $F_1(y,\m)$ has two non-degenerate critical points when $\m\ll y$ and no critical points when $\m\gg y$. When $\m \approx y$, the critical points can become degenerate although the third derivative of the phase function does not vanish there.
The difficulty here is to deal with them in a uniform way. We will use the following lemma with $a=y-\m$, $b=y$, $f(w)=\sin w-w$.

\begin{Lemma}\label{appstphasevai}
Let $\chi\in C_c^{\infty}((-3,3))$. Let $f\in C^{\infty}([-3,3];\R)$ such that $f^{(j)}(0)=0$ for $j=0,1,2$, $\mp f''(w)\leq 0$ and $c_1|w|\leq |f''(w)|\leq c_2|w|$ for $0\leq \pm w\leq 3$ with a constant $0<c_1<c_2$. For $a\in \R$, $b\geq 1$ and $j\in\mathbb{N}$, set
\begin{align}\label{appI_jdef}
I_j(a,b)=\int_{\R}\chi(w)w^j{\rm e}^{{\rm i}aw+{\rm i}bf(w)}{\rm d}w.
\end{align}

\begin{itemize}\itemsep=0pt
\item[$(i)$] Suppose $a,b\geq 1$, \smash{$|a|^{\frac{1}{2}}b^{-\frac{1}{2}}\leq 1$}, \smash{$|a|^{\frac{3}{2}}b^{-\frac{1}{2}}\geq 1/8$}. Then the map $[-3,3]\ni w\mapsto aw+bf(w)$ has just two critical points $w_{\pm}(a,b)$ with $\pm w_{\pm}(a,b)>0$. Moreover, we can write
\begin{align*}
I_0(a,b)=\sum_{\pm}I_{0,\pm}(a,b){\rm e}^{{\rm i} a w_{\pm}(a,b)+{\rm i}bf(w_{\pm}(a,b))},\qquad |\pa_a^{\a}I_{0,\pm}(a,b)|\leq C_{\a}a^{-\frac{1}{4}-\a}b^{-\frac{1}{4}}
\end{align*}
uniformly in $a,b\geq 1$, \smash{$|a|^{\frac{1}{2}}b^{-\frac{1}{2}}\leq 1$}, \smash{$|a|^{\frac{3}{2}}b^{-\frac{1}{2}}\geq 1/8$}.

\item[$(ii)$] For each $j\in\mathbb{N}$ and $N>0$, there exist $C_j>0$ and $C_{jN}>0$ such that
\begin{align*}
|I_j(a,b)|\lesssim \begin{cases}
C_jb^{-\frac{j+1}{3}}, & a\in \R,\  b\geq 1,\  |a|^{\frac{1}{2}}b^{-\frac{1}{2}}\leq 1, \\
C_{jN}\bigl(|a|^{\frac{3}{2}}b^{-\frac{1}{2}}\bigr)^{-N}|a|^{-j-\frac{1}{4}}b^{-\frac{1}{4}},&  a\leq -1,\  b\geq 1,\
  |a|^{\frac{3}{2}}b^{-\frac{1}{2}}\geq 1/8.
\end{cases}
\end{align*}
\end{itemize}
\end{Lemma}

\begin{proof}
When $a=0$, then the claim in $(ii)$ directly follows from the stationary phase theorem (see Lemma \ref{stphaselem} with $k=3$). Hence in the following, we assume $a\neq 0$.
Set
\begin{gather*}
s=|a|^{\frac{1}{2}}b^{-\frac{1}{2}},\qquad \l=|a|^{\frac{3}{2}}b^{-\frac{1}{2}},\\ F(w,s)=(\sgn a)w+b\l^{-1}f(sw)=(\sgn a)w+s^{-3}f(sw).
\end{gather*}
By the change of variable $w\mapsto sw$, we have
\begin{align}\label{appI_jscaling}
I_j(a,b)=s^{j+1}\int_{\R}\chi(sw)w^j{\rm e}^{{\rm i}\l F(w,s)}{\rm d}w.
\end{align}
By the assumption on $f$, we have $-\frac{c_2}{2}w^2\leq f'(w)\leq -\frac{c_1}{2}w^2$ for $|w|\leq 3$.

We note that there is $C_{\a}>0$ such that $|\pa_{w}^{\a}F(w,s)|\leq C_{\a}$ for $\a\geq 2$ uniformly as long as~${s\leq 1}$. Moreover,
\begin{align}\label{apprescalef'}
|\pa_wF(w,s)|=\big|(\sgn a)+s^{-2}f'(sw)\big|\geq c\bigl(1+|w|^2\bigr)\qquad \text{for}\quad |w|\in \bigl[2c_1^{-\frac{1}{2}},\infty\bigr)
\end{align}
uniformly in $a$. Moreover, when $a\geq 0$, the derivative $\pa_wF(w,s)=1+s^{-2}f'(sw)$ has just the two critical points $w_{\pm}(s)$ with
\[
\pm w_{\pm}(s)\in \bigl[\sqrt{2c_2^{-1}},\sqrt{2c_1^{-1}}\bigr] \qquad \text{and} \qquad c_1 \sqrt{2c_2^{-1}}\leq\big|\pa_w^2f(w_{\pm}(s),s)\big| \leq c_2 \sqrt{2c_1^{-1}}.
\]
 When $a<0$, then $F(\cdot,s)$ has no critical points. Moreover, $w_{\pm}(a,b):=sw_{\pm}(s)$ are critical points of the map $w\mapsto aw+bf(w)$.

 $(i)$ Suppose $a,b\geq 1$, $0<s\leq 1$ and $\l\geq 1/8$.
Let $\chi_+\in C_c^{\infty}((0,\infty);[0,1])$ such that
\[
\chi_+(w)=1 \qquad \text{for} \quad w\in \bigl[\sqrt{2c_2^{-1}},\sqrt{2c_1^{-1}}\bigr], \quad \supp \chi_+\subset \bigl(0,2\sqrt{c_1^{-1}}\bigr].
\]
 Set $\chi_-(w):=\chi_+(-w)$ and $\chi_0=1-\chi_+-\chi_-$.
We write \eqref{appI_jscaling} for $j=0$ as
\begin{gather*}
I_0(a,b)= \sum_{\pm}{\rm e}^{{\rm i}\l F(w_{\pm}(s),s)}I_{0,\pm}(a,b)+R_1(a,b),\\
I_{0,\pm}(a,b):=s\int_{\R}\chi(sw)\chi_{\pm}(w){\rm e}^{{\rm i}\l (F(w,s)-F(w_{\pm}(s),s))}{\rm d}w,\\
 R_1(a,b):=s\int_{\R}\chi(sw)\chi_{0}(w){\rm e}^{{\rm i}\l F(w,s)}{\rm d}w.
\end{gather*}

We prove that for $\a,\b,N\in\mathbb{N}$,
\begin{align}\label{I0abest}
\big|\pa_{\l}^{\a}(s\pa_{s})^{\b}I_{0,\pm}(a,b)\big|\lesssim s\l^{-\frac{1}{2}-\a},\qquad \big|\pa_{\l}^{\a}(s\pa_{s})^{\b}R_1(a,b)\big|\lesssim s\l^{-N}.
\end{align}
First, we observe
\begin{align}\label{wpmder}
|(s\pa_s)^{\a}w_{\pm}(s)|\lesssim 1\qquad \text{for}\quad 0<s\leq 1.
\end{align}
The estimate for $\a=0$ follows from the fact \smash{$\pm w_{\pm}(s)\in \bigl[\sqrt{2c_2^{-1}},\sqrt{2c_1^{-1}}\bigr]$} as we have proved.
To see \eqref{wpmder} for $\a\geq 1$, we recall $w_{\pm}(s)$ are the critical points of $F(\cdot,s)$, that is,
\begin{align*}
1+s^{-2}f'(sw_{\pm}(s))=0\Leftrightarrow f'(sw_{\pm}(s))=s^2.
\end{align*}
Thus we have
\begin{align}\label{wpdereq1}
f''(sw_{\pm}(s))(w_{\pm}(s)+s\pa_sw_{\pm}(s))=2s.
\end{align}
Since $|f''(sw_{\pm}(s))|\sim s|w_{\pm}(s)|\sim s$, we have $s|w_{\pm}(s)+s\pa_sw_{\pm}(s)|\lesssim 2s$, which implies $|s\pa_sw_{\pm}(s)|\allowbreak\lesssim 1$. This proves \eqref{wpmder} for $\a=1$. Differentiating \eqref{wpdereq1} many times and using induction on $\a$, we obtain \eqref{wpmder} for general $\a\geq 0$. By Taylor's theorem and the assumption on $f$, the function~${F(w,s)=w+s^{-3}f(sw)}$ is smooth with respect to $w$ and $s\in [0,1]$, where the point is that $F$ is smooth even near $s=0$. It turns out from this fact, \eqref{wpmder} and the Taylor's theorem that
\begin{align*}
F(w+w_{\pm}(s),s)-F(w_{\pm}(s),s)=g_{\pm}(w,s)w^2,
\end{align*}
where $g_{\pm}$ is smooth with respect to $w\in \supp \chi_{\pm}$ and $s\in (0,1]$ and satisfies
\begin{align*}
\big|\pa_w^{\a}(s\pa_s)^{\b}g_{\pm}(w,s)\big|\lesssim 1\qquad \text{for}\quad w\in \supp \chi_{\pm},\quad s\in (0,1].
\end{align*}
Moreover, we have $g_{\pm}(0,w)=\bigl(\pa_w^2F\bigr)(w_{\pm}(s),s)=s^{-1}f''(sw_{\pm}(s))$, which satisfies $|g_{\pm}(0,s)|\gtrsim 1$ uniformly in $0<s\leq 1$. Consequently, $I_{0,\pm}(a,b)$ can be written as
\begin{align*}
I_{0,\pm}(a,b)=s\int_{\R}\g_{0,\pm}(w,s){\rm e}^{{\rm i}\l g_{\pm}(w,s)w^2}{\rm d}w,
\end{align*}
where $\g_{0,\pm}(w,s)=\chi(s(w+w_{\pm}(s)))\chi_{\pm}(w+w_{\pm}(s))$. We also note that
\begin{align*}
\big|\pa_w^{\a}(s\pa_s)^{\b}\g_{0,\pm}(w,s)\big|\lesssim 1\qquad \text{for}\quad w\in\R,\quad s\in (0,1].
\end{align*}
Now we prove the first estimate of \eqref{I0abest}. For simplicity, we consider the case $(\a,\b)=(1,0)$ or $(\a,\b)=(0,1)$ only. To do this, we write
\begin{gather*}
\pa_{\l}I_{0,\pm}(a,b)={\rm i}s\int_{\R}\g_{0,\pm}(w,s)g_{\pm}(w,s)w^2{\rm e}^{{\rm i}\l g_{\pm}(w,s)w^2}{\rm d}w,\\
s\pa_{s}I_{0,\pm}(a,b)=s\int_{\R}\bigl(\g_{0,\pm}(w,s)\bigl(1+{\rm i}\l s\pa_sg_{\pm}(w,s) w^2\bigr) +s\pa_s\g_{0,\pm}(w,s) \bigr) {\rm e}^{{\rm i}\l g_{\pm}(w,s)w^2}{\rm d}w.
\end{gather*}
Now we apply the stationary phase theorem (see Lemma \ref{stphaselem} with $k=2$, $j=0,2$, $x_0=0$) and obtain the first estimate of \eqref{I0abest} for $(\a,\b)=(1,0), (0,1)$. The estimates for its higher derivatives are similarly proved. To prove the second estimate of \eqref{I0abest}, we use the fact that~${\pa_wF(w,s)}$ does not vanish for $w\in \supp \chi_0\cap \supp \chi(s\cdot )$ with the uniform estimate \eqref{apprescalef'}.
 Using the integration by parts many times with \eqref{apprescalef'} with the estimates for its higher-order derivatives, we obtain the second estimate of \eqref{I0abest}.

Next, we show
\begin{align}\label{appI_0ader}
|\pa_{a}^{\c}I_{0,\pm}(a,b)|\leq C_{\c}a^{-\frac{1}{4}-\c}b^{-\frac{1}{4}},\qquad |\pa_{a}^{\c}R_1(a,b)|\leq C_{\c N}\bigl(a^{\frac{3}{2}}b^{-\frac{1}{2}}\bigr)^{-N} a^{-\frac{1}{4}-\c}b^{-\frac{1}{4}}
\end{align}
for $\c\in\mathbb{N}$ and $N>0$.
Set $I_{0,\pm}^{\c}(a,b)=\pa_a^{\c}I_{0,\pm}(a,b)$. We can deduce $\big|\pa_{\l}^{\a}(s\pa_{s})^{\b}I_{0,\pm}^{\c}(a,b)\big|\leq \smash{C_{\a\b}s^{1+\c}\l^{-\frac{1}{2}-\a-\c}}$ by \eqref{I0abest} and induction due to \smash{$I_{0,\pm}^{\c+1}=\pa_aI_{0,\pm}^{\c}$}, $\pa_a= (\pa_a\l)\pa_{\l}+(\pa_as)\pa_s= s\bigl(3/2\pa_{\l}+2^{-1}\l^{-1}s \pa_s\bigr)$ and $s\leq 1, \l\geq 1/8$. In particular, we obtain
\[
|\pa_a^{\c}I_{0,\pm}(a,b)|\leq C_{\c}s^{1+\c}\l^{-\frac{1}{2}-\c}=C_{\c}a^{-\frac{1}{4}-\c}b^{-\frac{1}{4}}.
\]
 The estimate for $R_1(a,b)$ is similarly proved.

Since $w_{\pm}(a,b)=sw_{\pm}(s)$, we have $\l f_{a,b}(w_{\pm}(s))=\l w_{\pm}(s)+bf(sw_{\pm}(s))=aw_{\pm}(a,b)+bf(w_{\pm}(a,b)) $ and hence
\begin{align*}
I_0(a,b)=\sum_{\pm}I_{0,\pm}(a,b){\rm e}^{{\rm i}(aw_{\pm}(a,b)+bf(w_{\pm}(a,b)) )}+R_1(a,b).
\end{align*}
Finally, we show that $R_1(a,b)$ can be absorbed into either $I_{0,+}$ or $I_{0,-}$. To do this, it suffices to show \smash{$\big|\pa_a^{\a}\bigl({\rm e}^{-{\rm i}(aw_{\pm}(a,b)+bf(w_{\pm}(a,b)) )} R_1(a,b)\bigr)\big|\leq C_{\a N}a^{-\frac{1}{4}-\a}b^{-\frac{1}{4}}$}. We observe
\begin{align*}
\big|\pa_a^{\a}{\rm e}^{{\rm i}\l f_{a,b}(w_{\pm}(s))}\big|\leq C_{\a}s^{-\a}=C_{\a}a^{-\frac{\a}{2}}b^{\frac{\a}{2}}.
\end{align*}
since $\pa_a (\l f_{a,b}(w_{\pm}(s)))=s^{-1}w_{\pm}(s)$. Thus the second inequality of \eqref{appI_0ader} with $N\gg 1$ gives the desired estimate.

 $(ii,\,1)$ First, we consider the case $a\in \R$, $b\geq 1$ and $s\leq 1$.
Let $\chi_1\in C_c^{\infty}(\R;[0,1])$ such that
\[
\chi_1(w)=1 \qquad \text{for} \ |w|\leq \sqrt{2c_1^{-1}},\qquad \chi_1(w)=0  \quad \text{for} \ |w|\geq 2\sqrt{c_1^{-1}}.
\] Set $\chi_2=1-\chi_1$. We write~\eqref{appI_jscaling} for $j=0$ as
\begin{align*}
I_j(a,b)=&s^{j+1}\int_{\R}\chi(sw)w^j(\chi_1(w)+\chi_2(w)) {\rm e}^{{\rm i}\l f_{a,b}(w)}{\rm d}w=:I_{j,1}(a,b)+I_{j,2}(a,b).
\end{align*}
The second term is easy to handle: The integration by parts with \eqref{apprescalef'} yields $|I_{j,2}(a,b)|\leq\smash{ C_{jN}s^{j+1}\l^{-N}}$ for $N>0$. Taking $N=\frac{j+1}{3}$, we have \smash{$|I_{j,2}|\lesssim s^{j+1}\l^{-\frac{j+1}{3}}=b^{-\frac{j+1}{3}}$}. Thus, we focus on the estimate for $I_{1,j}$. When $\l\leq 1$ (which is equivalent to \smash{$|a|\leq b^{\frac{1}{3}}$}), then $|I_{j,1}(a,b)|\lesssim\smash{ s^{j+1}=|a|^{\frac{j+1}{2}}b^{-\frac{j+1}{2}}\leq b^{-\frac{j+1}{3}}}$. On the other hand, when $\l\geq 1$, then the stationary phase theorem (see Lemma \ref{stphaselem} with $k=3$) implies \smash{$|I_{j,1}(a,b)|\lesssim s^{j+1}\l^{-\frac{j+1}{3}}=b^{-\frac{j+1}{3}}$}. Thus we obtain $|I_j(a,b)|\lesssim \smash{b^{-\frac{j+1}{3}}}$.

 $(ii,\,2)$ Suppose $a\leq -1$, $b\geq 1$, $s\leq 1$ and $\l\geq 1/8$. The case $s\geq 1$ is dealt with later.
Since~${\l\geq 1/8}$, it suffices to prove the inequality for large integer $N$.
By the assumption on~$f$, we have $f_{a,b}'(w)\geq 1+ cw^2$ and \smash{$\big|f_{a,b}^{(\a)}(w)/f_{a,b}'(w)\big|\leq C_{\a}(1+|w|)^{-1}$} for $sw\in \supp \chi$ with a~constant~$C_{\a}$ independent of $s,\l$ for $\a\geq 2$. In fact,
\begin{align*}
&\big|f_{a,b}''(w)\big|=b\l^{-1}s\big|f''(sw)\big|\leq b\l^{-1}s^2\cdot |sw|\lesssim (1+ |w|),\\
&\big|f^{(\a)}_{a,b}(w)\big|=b\l^{-1}s^{\a}\big|f^{(\a)}(sw)\big|\leq b\l^{-1}s^3\big|f^{(\a)}(sw)\big|=O(1)
\end{align*}
for $\a\geq 3$.

Now we set $L=D_w\circ(f_{s,\l}'(w))$. Then the integration by parts yields
\begin{align*}
I_j(a,b)= \l^{-N}s^{j+1}\int_{\R}L^{N_1}\bigl(\chi(sw)w^j\bigr){\rm e}^{{\rm i}\l f_{a,b}(w)}{\rm d}w,
\end{align*}
for $N_1\in\mathbb{N}$. Taking $N_1$ large enough, we have
\begin{align*}
\big|L^{N_1}\bigl(\chi(sw)w^j\bigr)\big|\leq C\bigl(1+w^2\bigr)^{-1}
\end{align*}
with a constant $C>0$ independent of $s$, $\l$ and $w\in \supp \chi$. Here the independence with respect to $s$, $\l$ follows from the estimates for $f_{a,b}(w)$. Thus we have $|I_j(a,b)|=O\bigl(\l^{-N}s^{j+1}\bigr)$.
Taking~${N_1=N+j+\frac{1}{2}}$, we have
\begin{align*}
\l^{-N_1}s^{j+1}=\bigl(|a|^{\frac{3}{2}}b^{-\frac{1}{2}}\bigr)^{-N}\cdot |a|^{-\frac{3}{2}j-\frac{3}{4}}b^{\frac{1}{2}j+\frac{1}{4}} \cdot |a|^{\frac{j+1}{2}}b^{-\frac{j+1}{2}}=\bigl(|a|^{\frac{3}{2}}b^{-\frac{1}{2}}\bigr)^{-N}|a|^{-j-\frac{1}{4}}b^{-\frac{1}{4}}.
\end{align*}

 $(ii,\,3)$ Suppose $a\leq -1$, $b\geq 1$, $s\geq 1$ and $\l\geq 1/8$. To deal with this case, we do not use~\eqref{appI_jscaling} but the definition \eqref{appI_jdef}. Since $a\leq -1$ and $f'(w)<0$, we have $|\pa_w(aw+bf(w))|\geq a$. Thus the integration by parts yields $|I_j(a,b)|\leq C_{N'}|a|^{-N'}$ for all $N'>0$. Now take $N'>0$ large enough~such that \smash{$|a|^{-N'}\leq \bigl(|a|^{\frac{3}{2}}b^{-\frac{1}{2}}\bigr)^{-N}|a|^{-j-\frac{1}{4}}b^{-\frac{1}{4}}$}, which is possible since $s\geq 1$ implies~${|a|\geq b}$. This completes the proof.
\end{proof}

\begin{proof}[Proof of Proposition \ref{Besselasymp}]
Set $a=y-\m$, $b=y$ and $f(w)=\sin w-w$. Then it turns out that $aw+bf(w)=y\sin w-\m w$ and $f$ satisfies the assumption of Lemma \ref{appstphasevai}. We observe
\begin{align*}
&\pa_{\m}=(\pa_{\m} a)\pa_{a}+(\pa_{\m} b)\pa_{b}=-\pa_a,\qquad |a|^{\frac{3}{2}}b^{-\frac{1}{2}}=|\m-y|^{\frac{3}{2}}y^{-\frac{1}{2}},\\
&s\leq 1\Leftrightarrow |\m|\leq 2|y|,\qquad \l\geq 1/8\Leftrightarrow |y-\m|\geq \frac{1}{2}y^{\frac{1}{3}},\\
&\pa_{\m}^{\a}I_0(a,b)=(-{\rm i})^{\a}\int_{\R}{\rm e}^{{\rm i}y \sin w-{\rm i}\m w }w^{\a}\chi_1(w){\rm d}w.
\end{align*}
Then the part (ii) follows from Lemmas \ref{appBeselreduce} and \ref{appstphasevai}\,(ii) when $\m\leq 2y$ holds. The case~${\m\geq 2y}$ follows from the integration by parts in the above expression of $\pa_{\m}^{\a}I_0(a,b)$, where we use $|\pa_y(y \sin w-\m w)|\geq \m-|y|\geq \frac{1}{4}(\m+y)$.

Since $w_{\pm}(a,b)$ are critical points of $aw+bf(w)$, that is, $w_{\pm}(a,b)=\pm\cos^{-1}(\frac{b-a}{b})$, we have
\begin{align*}
aw_{\pm}(a,b)+bf(w_{\pm}(a,b))=\pm yh_1\left(\frac{\m}{y}\right),
\end{align*}
where we recall $h_1(z)=\sqrt{1-z^2}-z\cos^{-1}z$. This proves the part (i) if we set
\begin{align*}
a_{+,y}(\m)=\frac{1}{2\pi}I_{0,+}(a,b)+ {\rm e}^{-yh_1(\frac{\m}{y})} R_0(\m,y),\qquad a_{-,y}(\m)=\frac{1}{2\pi}I_{0,-}(a,b).\tag*{\qed}
\end{align*}\renewcommand{\qed}{}
\end{proof}

\subsection{Asymptotics of the Beta function}\label{subsecasymbeta}

For $m\in [1,\infty)$ and $\n>0$, we define
\begin{align}\label{Fnu1def}
F_{\n,1}(m):=\n^{-1}\frac{\Gamma\bigl(\n+\frac{1}{2}\bigr)}{\sqrt{\pi}\Gamma(\n)}\cdot \frac{\Gamma(m+2\n)}{\Gamma(m+1)\Gamma(2\n)}.
\end{align}
By \cite[Fact 4.8, equation~(4.30)]{BKO}, we have $C_m^{\n}(1)=\frac{\Gamma(m+2\n)}{\Gamma(m+1)\Gamma(2\n)}$ and hence
\begin{align}\label{Fnu1Gegen}
F_{\n,1}(m)=\n^{-1}\frac{\Gamma\bigl(\n+\frac{1}{2}\bigr)C_m^{\n}(1)}{\sqrt{\pi}\Gamma(\n)}\qquad \text{for}\quad m\in\mathbb{N}.
\end{align}
We denote the Beta function by $B(x,y)$. Then it is known that $B(x,y)=\frac{\Gamma(x)\Gamma(y)}{\Gamma(x+y)}$. We may write
\begin{align}\label{Fnu1beta}
F_{\n,1}(m)=\n^{-1}\frac{\Gamma\bigl(\n+\frac{1}{2}\bigr)}{\sqrt{\pi}\Gamma(\n)}\cdot \frac{1}{mB(m,2\n)}.
\end{align}

\begin{Proposition}\label{Fnuestimate}
Let $\a\geq 0$ be an integer and $\n>0$. Then there exists $C_{\n,\a}>0$ such that
\begin{align*}
|\pa_m^{\a}B(m,2\n)|\leq C_{\n,\a}m^{-2\n-\a},\qquad |\pa_m^{\a}F_{\n,1}(m)|\leq C_{\n,\a}m^{2\n-1-\a}
\end{align*}
for all $m\in [1,\infty)$.
\end{Proposition}

\begin{proof}First, we note that it suffices to prove these estimates for sufficiently large $m$. As is well known (essentially due to Stirling's formula), we have $B(m,2\n)\sim \Gamma(2\n)m^{-2\n}$ as $m\to \infty$.
This implies that the estimates for $F_{\n,1}(m)$ follow from the estimates for $B(m,2\n)$ and \eqref{Fnu1beta}. The case $\a=0$ follows from $B(m,2\n)\sim \Gamma(2\n)m^{-2\n}$ as $m\to \infty$. Hence, in the following, we consider the estimates for $\a\geq 1$.

Let $\g\in C^{\infty}(\R;[0,1])$ satisfy $\g(t)=1$ for $t\geq 1/2$ and $\g(t)=0$ for $t\leq \frac{1}{4}$.
Since $B(m,2\n)=\int_0^1t^{m-1}(1-t)^{2\n-1}{\rm d}t$, we have
\begin{align*}
\pa_m^{\a}B(m,2\n)=\int_0^1t^{m-1}g_{\a,1}(t){\rm d}t+\int_0^1t^{m-1}g_{\a,2}(t){\rm d}t,
\end{align*}
where we set $g_{\a,1}(t)=(\log t)^{\a}(1-t)^{2\n-1}\g(t)$ and $g_{\a,2}(t)=(\log t)^{\a}(1-t)^{2\n-1}(1-\g(t))$.

First, we deal with the second term. Since $\supp g_{\a,2}\subset \{t\leq 1/2\}$ and since $|(\log t)^{\a}|\lesssim |t|^{-1}$ for $0<t\leq 1/2$, we have
\begin{align*}
\left|\int_0^1t^{m-1}g_{\a,2}(t){\rm d}t\right|\lesssim \int_0^{\frac{1}{2}}t^{m-2}{\rm d}t\lesssim 2^{-m}\lesssim m^{-2\n-\a}
\end{align*}
for sufficiently large $m\geq 1$.

Next, we consider the first term. We note
\begin{align*}
\big|\pa_{t}^{\b}g_{\a,1}(t)\big|\leq C_{\a}|1-t|^{2\n-1+\a-\b}\qquad\text{for}\quad t\in [0,1].
\end{align*}
Now we write
\begin{align*}
t^{m-1}={\rm e}^{(m-1)(\log t)}={\rm e}^{{\rm i}(m-1)\tilde{f}(t)},\qquad \tilde{f}(t)=-{\rm i}\log t.
\end{align*}
Then $\tilde{f}'(t)=-{\rm i}/t\neq 0$ and $\im \tilde{f}(t)\geq 0$ hold for $t\in \supp \g\cap (-\infty,1]$.
Now it follows from Lemma \ref{singularasym} with $\l=m$, $\m=2\n-1$ and the change of variable $x=1-t$ that
\begin{align*}
\left|\int_0^1t^{m-1}g_{\a,1}(t){\rm d}t\right|\lesssim m^{-2\n-\a}
\end{align*}
holds for sufficiently large $m\geq 1$. This completes the proof.
\end{proof}

\subsection{Asymptotics of the Gegenbauer polynomials}\label{appasymgegensubsec}

In this appendix, we prove Proposition \ref{Gegenasymp} using non-stationary phase theorem. The case~${\n=0}$ is easy to deal with (see the proof of Proposition \ref{Gegenasymp} below) and hence we focus on the case~${\n>0}$.

By \cite[equation~(4.30)]{BKO} (see also \cite[Section 10]{IJ}) and \eqref{Fnu1Gegen},
\begin{align}\label{Gegenintrep}
\n^{-1}C_m^{\n}(\cos\f)=&F_{\n,1}(m)\int_{-1}^{1}(\cos\f+{\rm i}(\sin\f)u)^m\bigl(1-u^2\bigr)^{\n-1}{\rm d}u
\end{align}
for $\n>0$, $m\in\mathbb{N}^*$ and $\f\in [0,\pi]$, where $F_{\n,1}$ is defined in \eqref{Fnu1def}. To get an asymptotics of~$C_m^{\n}(\cos\f)$, we extend $\mathbb{N}^*\ni m\mapsto C_m^{\n}(\cos\f)$ to a function on $m\in [1,\infty)$ (up to a negligible term). However, since the integrant $(\cos\f+{\rm i}(\sin\f)u)^m$ is not a single-valued function, we have to do it a bit carefully.

First, we deal with the case $\f\in \bigl(\frac{\pi}{4}, \frac{3}{4}\pi\bigr)$.
We consider two branches of logarithmic smooth functions $\log^+z\colon\mathbb{C}\setminus \{{\rm i}y\mid y\leq 0\}\to \mathbb{C}$ and $\log^-z\colon\mathbb{C}\setminus \{{\rm i}y\mid y\geq 0\}\to \mathbb{C}$ such that $\log^+(\cos \f+{\rm i}\sin \f)={\rm i}\f$ and $\log^-(\cos\f-{\rm i}\sin\f))=-{\rm i}\f$ for all $\f\in [0,\pi]$. We define smooth functions~$f_{\pm}$~by
\begin{align*}
f_{\pm}(\f,u)=-{\rm i}\log^{\pm}(\cos\f+{\rm i}(\sin\f)u)\mp \f
\end{align*}
for $(\f,u)$ such that $\cos\f+{\rm i}(\sin\f)u$ belongs to the domain of $\log^{\pm}$. We note that
$([0,\pi]\times \R_{\pm})\cup \bigl(\bigl(v[0,\frac{\pi}{4}\bigr]\cup \bigl[\frac{3}{4}\pi,\pi\bigr]\bigr)\times \bigl[-\frac{1}{2},\frac{1}{2}\bigr] \bigr)$ is included in the domain of $f_{\pm}$.
Then ${\rm e}^{{\rm i}mf_{\pm}(\f,u)}=(\cos\f+{\rm i}(\sin\f)u)^m{\rm e}^{\mp {\rm i}m\f}$ for all $m\in\mathbb{N}^*$ and $\pm u>0$.
Now let $\chi_{\pm},\chi_0\in C^{\infty}(\R;[0,1])$ such that $\chi_++\chi_0+\chi_-=1$ on $\R$, $\supp\chi_0\subset \bigl[-\frac{1}{2},\frac{1}{2}\bigr]$, $\supp \chi_+\subset \bigl[\frac{1}{4},\infty\bigr)$ and $\supp\chi_-\subset \bigl(-\infty,-\frac{1}{4}\bigr]$. We define
\begin{align*}
&H_{\n,\pm}(m,\f):=\int_{-1}^1{\rm e}^{{\rm i}mf_{\pm}(\f,u)}\chi_{\pm}(u)\bigl(1-u^2\bigr)^{\n-1}{\rm d}u,\\
&E_{\n}(m,\f):=\int_{-1}^1(\cos \f+{\rm i}(\sin\f)u)^m\chi_{0}(u)\bigl(1-u^2\bigr)^{\n-1}{\rm d}u,
\end{align*}
where $H_{\n,\pm}(m,\f)$ is defined for all $m\in [1,\infty)$ although $E_{\n}(m,\f)$ is defined only for $m\in \mathbb{N}^*$ whenever $\f\in \bigl(\frac{\pi}{4}, \frac{3}{4}\pi\bigr)$ (note that $\cos\f+{\rm i}(\sin \f)u$ may be zero for $u\in \supp \chi_0$). By these construction and \eqref{Gegenintrep}, we have
\begin{align}\label{Gegendecom}
\n^{-1}C_m^{\n}(\cos\f)=&F_{\n,1}(m)\bigl({\rm e}^{{\rm i}m\f}H_{\n,+}(m,\f)+{\rm e}^{-{\rm i}m\f}H_{\n,-}(m,\f)+E_{\n}(m,\f) \bigr)
\end{align}
for $m\in\mathbb{N}^*$ and $\f\in \bigl(\frac{\pi}{4}, \frac{3}{4}\pi\bigr)$.

Next, we consider the case $\f\in \bigl[0,\frac{\pi}{4}\bigr]\cup \bigl[\frac{3}{4}\pi,\pi\bigr]$. We define
\begin{gather*}
E_{\n}'(m,\f):=\int_{-1}^1{\rm e}^{{\rm i}mf_{+}(\f,u)}\chi_{0}(u)\bigl(1-u^2\bigr)^{\n-1}{\rm d}u,\qquad m\in [1,\infty),\\ \f\in \left[0,\frac{\pi}{4}\right]\cup \left[\frac{3}{4}\pi,\pi\right],
\end{gather*}
where we note that $\cos \f+{\rm i}(\sin\f)u\in \mathbb{C}\setminus \{{\rm i}y\mid y\leq 0\}$ whenever $\f\in \bigl[0,\frac{\pi}{4}\bigr]\cup \bigl[\frac{3}{4}\pi,\pi\bigr]$ and $u\in \supp \chi_0$. Then
\begin{align}\label{Gegendecom2}
\n^{-1}C_m^{\n}(\cos\f)=&F_{\n,1}(m)\bigl({\rm e}^{{\rm i}m\f}H_{\n,+}(m,\f)+{\rm e}^{-{\rm i}m\f}H_{\n,-}(m,\f)+{\rm e}^{{\rm i}m\f}E_{\n}'(m,\f) \bigr)
\end{align}
holds for $m\in\mathbb{N}^*$ and $\f\in \bigl[0,\frac{\pi}{4}\bigr]\cup \bigl[\frac{3}{4}\pi,\pi\bigr]$.

Now we show that $H_{\n,\pm}$, $E_{\n}$ and $E_{\n}'$ satisfy symbol-type estimates. The basic idea of the proof for $H_{\n,\pm}$ and $E_{\n}$ is to use the non-stationary phase theorem with a singular amplitude (see Lemma~\ref{singularasym}). To do this, we observe that our phase function $f_{\pm}(\f,u)$ satisfies $|f_{\pm}'(\f,u)|\gtrsim |\sin \f|$ on~${u\in \supp \chi_{\pm}}$ (see the proof below). Although $\pa_u f_{\pm}(\f,u)=0$ when $\f=0,\pi$, we consider a~large parameter $\l=m\sin \f$ rather than $m$ and can apply Lemma \ref{singularasym}.

\begin{Lemma}\label{Gegenlemma}
Let $\a\geq 0$ be an integer and $\n,N>0$.
\begin{itemize}\itemsep=0pt
\item[$(i)$] There exists $C_{N,\n}>0$ such that $|E_{\n}(m,\f)|\leq C_{N,\n}m^{-N}$ for $m\in\mathbb{N}^*$ and $\f\in \bigl(\frac{\pi}{4}, \frac{3}{4}\pi\bigr)$.

\item[$(ii)$] There exists $C_{\a,N,\n}>0$ such that $|\pa_{m}^{\a}E_{\n}'(m,\f)|\leq C_{\a,N,\n}m^{-\a}(1+m\sin\f)^{-N}$ for $m\in[1,\infty)$ and $\f\in \bigl[0,\frac{\pi}{4}\bigr]\cup \bigl[\frac{3}{4}\pi,\pi\bigr]$.

\item[$(iii)$] There exists $C_{\a,\n}>0$ such that $|\pa_{m}^{\a}H_{\n,\pm}(m,\f)|\leq C_{\a,\n}m^{-\a}(1+m\sin\f)^{-\n}$ for $m\in[1,\infty)$ and $\f\in [0,\pi]$.
\end{itemize}
\end{Lemma}

\begin{proof}
 $(i)$ Since \smash{$|\cos\f+{\rm i}(\sin \f)u|\leq \sqrt{\frac{5}{8}}$} for $u\in \supp \chi_0$ and $\f\in \bigl(\frac{\pi}{4}, \frac{3}{4}\pi\bigr)$, we have
 \[
|\cos\f+{\rm i}(\sin \f)u|^m\leq \left(\frac{5}{8}\right)^{\frac{m}{2}}\lesssim m^{-N}
\]
 for $m\in\mathbb{N}^*$. The estimate for $E_{\n}$ directly follows from this inequality.

 $(ii)$ Setting $\g(u):=\chi_0(u)\bigl(1-u^2\bigr)^{\n-1}$, we write $E_{\n}'(m,\f)=\int_{\R}{\rm e}^{{\rm i}mf_{+}(\f,u)}\g(u){\rm d}u$.

First, we consider the case $\a=0$ by using the integration by parts.
Since $\log^+(\cos \f+{\rm i}\sin \f)={\rm i}\f$, we have
\begin{align}\label{f_+rep}
f_+(\f,u)={\rm i}\int_u^1\frac{{\rm d}}{{\rm d}r}\log^+(\cos \f+{\rm i}(\sin \f) r){\rm d}r=-\int_u^1\frac{\sin\f(\cos \f-{\rm i}(\sin \f) r)}{\cos^2 \f+r^2\sin^2 \f }{\rm d}r
\end{align}
and hence \smash{$\im f_+(\f,u)=\int_u^1\frac{(\sin^2 \f) r}{\cos^2 \f+r^2\sin^2 \f }{\rm d}r$}. This implies $\im f_+(\f,u)\geq 0$ for $u\in [-1,1]$. In fact, this is trivial for $u\geq 0$. For $u\leq 0$, we write
\begin{align*}
\int_u^1\frac{\bigl(\sin^2 \f\bigr) r}{\cos^2 \f+r^2\sin^2 \f }{\rm d}r={}&\int_0^1\frac{\bigl(\sin^2 \f\bigr) r}{\cos^2 \f+r^2\sin^2 \f }{\rm d}r+\int_u^0\frac{\bigl(\sin^2 \f\bigr) r}{\cos^2 \f+r^2\sin^2 \f }{\rm d}r\\
={}&\int_{-u}^1\frac{\bigl(\sin^2 \f\bigr) r}{\cos^2 \f+r^2\sin^2 \f }{\rm d}r,
\end{align*}
 which is non-negative.

Now we set $L=\pa_u\circ ({\rm i}\pa_uf_+(\f,u))^{-1}$. By integrating by parts, we have
\[
|E_{\n}'(m,\f)|=m^{-N}\left|\int_{\R}{\rm e}^{{\rm i}mf_{+}(\f,u)} L^N\g(u){\rm d}u\right|\leq m^{-N}\int_{\R}\big| L^N\g(u)\big|{\rm d}u
\]
 by virtue of $\im f_+(\f,u)\geq 0$. Since $\pa_u^{\a}f_+(\f,u)=(-1)^{\a-1}(\a-1)!(\sin\f)^{\a}(\cos \f+{\rm i}u\sin \f )^{-\a}$, for each $\a\geq 2$, we obtain $|\pa_uf_+(\f,u)|^{-1}\lesssim |\sin \f|^{-1}$ and $|\pa_u^{\a}f_+(\f,u)||\pa_uf_+(\f,u)|^{-1} \lesssim 1$ for $u\in [-1,1]$ and $\f\in \bigl[0,\frac{\pi}{4}\bigr]\cup \bigl[\frac{3}{4}\pi,\pi\bigr]$. By Lemma \ref{Leibnizrule}, we conclude $\big| L^N\g(u)\big|\lesssim |\sin\f|^{-N}$ and hence $|E_{\n}'(m,\f)|\lesssim m^{-N}|\sin\f|^{-N}$ for $\f\in \bigl[0,\frac{\pi}{4}\bigr]\cup \bigl[\frac{3}{4}\pi,\pi\bigr]$ (we note that $\g$ is compactly supported). On the other hand, the inequality $\im f_+(\f,u)\geq 0$ also implies $|E_{\n}'(m,\f)|\lesssim 1$. Combining these estimates, we have proved (ii) for $\a=0$.

Next, we consider the case $\a\geq 1$. We note
\begin{align*}
\pa_m^{\a}E'_{\n}(m,\f)={\rm i}^{\a}\int_{\R}{\rm e}^{{\rm i}mf_{+}(\f,u)}f_{+}(\f,u)^{\a}\g(u)du.
\end{align*}
As in the case of $\a=0$, we can deduce $\big|L^N(f_{+}(\f,u)^{\a}\g(u))\big|\lesssim |{\sin\f}|^{\a-N}$, which leads to the part (ii) for $\a\geq 1$ since $f_{+}(0,u)=f_{+}(\pi,u)=0$.

 $(iii)$ We deal with the case $+$ only. There exists $c>0$ such that $|{\cos\f+{\rm i}u\sin \f}|\geq c$ for~${u\in \supp\chi_+}$ and $\f\in [0,\pi]$.
 Since \smash{$\im f_+(\f,u)=\int_u^1\frac{(\sin^2 \f) r}{\cos^2 \f+r^2\sin^2 \f }{\rm d}r$}, which is proved above, we have $\im f_+(\f,u)\geq 0$ for $u\in \supp \chi_+$ and $\f\in [0,\pi]$. Moreover, we have $f_+(\f,1)=0$.

Now we prove $(iii)$ for $\a=0$. Set $m'=m\sin \f$ and $g(\f,u)=(\sin \f)^{-1}f_+(\f,u)$. Then we write \smash{$H_{\n,+}(m,\f)=\int_{-1}^1{\rm e}^{{\rm i}m'g(\f,u)}\chi_{+}(u)\bigl(1-u^2\bigr)^{\n-1}{\rm d}u$}. Since
\[
\pa_u^{\a}g(\f,u)=(-1)^{\a-1}(\a-1)!(\sin\f)^{\a-1}(\cos \f+{\rm i}u\sin \f )^{-\a}
\]
 for $\a\geq 1$ and since $|{\cos \f+{\rm i}u\sin \f}|\in \bigl[\frac{1}{\sqrt{2}},1\bigr]$ for $\f\in \bigl[0,\frac{\pi}{4}\bigr]\cup \bigl[\frac{3}{4}\pi,\pi\bigr]$, we have
\begin{align*}
|\pa_ug(\f,u)|\geq c,\qquad |\pa_u^{\a}g(\f,u)| |\pa_ug(\f,u)|^{-1}\leq C_{\a}'
\end{align*}
with a constant $c>0$ and $C_{\a}'>0$. By Lemma \ref{singularasym} with $\l=m\sin \f$ and the change of variable~${x=1-u}$, we obtain $|H_{\n,+}(m,\f)|\lesssim (m')^{-\n}$ for $m'=m\sin\f\geq 1$. On the other hand, it follows from $\im f_+(\f,u)\geq 0$ that $|H_{\n,+}(m,\f)|\lesssim 1$. Combining these estimates, we obtain~${|H_{\n,+}(m,\f)|\lesssim (1+m')^{-\n}=(1+m\sin\f)^{-\n}}$ for $m\sin \f\geq 1$. For $m\sin \f\leq 1$, this estimate is easy to prove.

Finally, we consider the case $\a\geq 1$. We observe that
\begin{align*}
\pa_m^{\a}H_{\n,+}(m,\f)=i^{\a}\int_{\R}{\rm e}^{{\rm i}m'g(\f,u)}f_{+}(\f,u)^{\a}\bigl(1-u^2\bigr)^{\n-1}\chi_+(u){\rm d}u
\end{align*}
and \smash{$\big|\pa_u^{\b}f_{+}(\f,u)^{\a}\big|\lesssim (\sin\f)^{\a}(1-u)^{\max(\a-\b,0)}$} since $f_+(\f,u)=O((1-u))$ by \eqref{f_+rep}. Thus, a~similar argument as the case $\a=0$ shows that $|\pa_m^{\a}H_{\n,+}(m,\f)|\lesssim (\sin\f)^{\a}(1+m\sin\f)^{-\n-\a}$. Using an inequality $\sin\f(1+m\sin\f)^{-1}\leq m^{-1}$, we obtain $(iii)$ for $\a\geq 1$.
\end{proof}

\begin{proof}[Proof of Proposition \ref{Gegenasymp}]
First, we consider the case $\n=0$. From \cite[equation~(4.28)]{BKO}, we have
\begin{align*}
\lim_{\n\to 0}\n^{-1}C_m^{\n}(\cos\f)=\frac{2\cos (m\f)}{m}=\frac{{\rm e}^{{\rm i}m\f}+{\rm e}^{-{\rm i}m\f}}{m}.
\end{align*}
Thus, we can take $g_{\n,\pm}(m,\f)=1/m$ and $r(m,\f)=0$.

Next, we consider the case $\n>0$. For $\f\in \bigl(\frac{\pi}{4}, \frac{3}{4}\pi\bigr)$, we set $g_{\n,\pm}(m,\f)=F_{\n,1}(m)H_{\n,\pm}(m,\f)$ and $r(m,\f)=F_{\n,1}(m)E_{\n}(m,\f)$. For $\f\in \bigl[0,\frac{\pi}{4}\bigr]\cup \bigl[\frac{3}{4}\pi,\pi\bigr]$, we set
\begin{align*}
g_{\n,+}(m,\f)=F_{\n,1}(m)(H_{\n,+}(m,\f)+E_{\n}'(m,\f)),\qquad g_{\n,-}(m,\f)=F_{\n,1}(m)H_{\n,-}(m,\f)
\end{align*}
and $r(m,\f)=0$. Then Proposition \ref{Gegenasymp} directly follows from the identities \eqref{Gegendecom}, \eqref{Gegendecom2}, Proposition \ref{Fnuestimate} and Lemma \ref{Gegenlemma}.
\end{proof}

\section{Finite time Strichartz estimates}\label{finitetime}

Here, we show that the Strichartz estimates hold for finite time $T$ assuming that they hold for each time $t$ with $|t|\leq T_0$.
The following argument is more or less well known.
For simplicity, assuming $T_0<T<2T_0$, \eqref{homStr} and \eqref{inhomStr} hold for $T_0$ and $T-T_0$, we shall show that \eqref{homStr} and~\eqref{inhomStr} hold for $T$.

Set $U(t)={\rm e}^{-{\rm i}tH_{k,a}}$.
Since $U(t+s)=U(t)U(s)$, $[-T,T]=[-T_0,T_0]\cup [-T,-T_0]\cup [T_0,T]$ and~${T<2T_0}$,
\begin{align*}
\|U(t)u_0\|_{L^p([-T,T];L^q)}^p={}&\|U(t)u_0\|_{L^p([-T_0,T_0];L^q)}^p+\|U(t)U(-T_0)u_0\|_{L^p([-T+T_0,0];L^q)}^p\\
&+\|U(t)U(T_0)u_0\|_{L^p([0,T-T_0];L^q)}^p\\
\lesssim{}&\|u_0\|_{L^2}^p+\|U(-T_0)u_0\|_{L^2}^p+\|U(T_0)u_0\|_{L^2}^p=3\|u_0\|_{L^2}^p,
\end{align*}
where we use the fact that $U(t)$ is unitary in the last line. This proves \eqref{homStr} for $T$.

Set $\Gamma f(t)=\int_0^tU(t-s)f(s){\rm d}s$.
To see that \eqref{inhomStr} holds for $T$, it suffices to prove
\begin{align}\label{appinhomaim}
\|\Gamma f\|_{L^{p_1}([-T,-T_0]\cup [T_0,T];L^{q_2})}\lesssim \|f\|_{L^{p_2^*}([-T,T];L^{q_2^*})}.
\end{align}
We firstly observe
\begin{align}\label{apphomdual}
\left\|\int_0^{T_0}U(-s)f(s){\rm d}s \right\|_{L^2}\leq C\|f\|_{L^{p_2^*}([0,T_0];L^{q_2^*})}
\end{align}
due to the duality of \smash{$\|U(t)u_0\|_{L^p([0,T_0];L^q)}\lesssim\|u_0\|_{L^2}$} which in turn follows from \eqref{homStr} for $T_0$. Setting \smash{$\Gamma_1f(t)=\int_0^{T_0}U(t-s)f(s){\rm d}s$} and $\Gamma_2f(t)=\int_{T_0}^tU(t-s)f(s){\rm d}s$, we have $\Gamma=\Gamma_1+\Gamma_2$. Since \smash{$\Gamma_1 f(t)=U(t)\int_0^{T_0}U(-s)f(s){\rm d}s$}, the homogeneous estimate \smash{$\|U(t)u_0\|_{L^p([-T,T];L^q)}\lesssim\|u_0\|_{L^2}$} implies
\begin{align*}
\|\Gamma_1 f\|_{L^{p_1}([T_0,T];L^{q_2})}\lesssim\left\|\int_0^{T_0}U(-s)f(s){\rm d}s \right\|_{L^2}\underbrace{\lesssim}_{\eqref{apphomdual}} \|f\|_{L^{p_2^*}([0,T_0];L^{q_2^*})}.
\end{align*}
On the other hand, setting $g(t)=f(t+T_0)$, we have $\Gamma_2f(t)=\int_0^{t-T_0}U(t-T_0-s)g(s){\rm d}s$. Using~\eqref{inhomStr} for $T-T_0(<T_0)$, we obtain
\begin{align*}
\|\Gamma_2f\|_{L^{p_1}([T_0,T];L^{q_1})}=\left\|\int_0^{t}U(t-s)g(s){\rm d}s\right\|_{L^{p_1}([0,T-T_0];L^{q_1})}\lesssim \|g\|_{L^{p_2^*}([-T+T_0,T-T_0];L^{q_2^*})}.
\end{align*}
Combining them, we conclude \smash{$\|\Gamma f\|_{L^{p_1}([T_0,T];L^{q_1})}\lesssim \|f\|_{L^{p_2^*}([-T,T];L^{q_2^*})}$}. Similarly, we have
\begin{align*}
\|\Gamma f\|_{L^{p_1}([-T,-T_0];L^{q_1})}\lesssim \|f\|_{L^{p_2^*}([-T,T];L^{q_2^*})}.
\end{align*}
Thus we have proved \eqref{appinhomaim}.

\subsection*{Acknowledgment}
KT was supported by JSPS KAKENHI Grant Number 23K13004, and HT was supported by JSPS KAKENHI Grant Numbers 20J00024 and 23K12947. HT is grateful to Toshiyuki Kobayashi for helpful comments. The authors would like to appreciate Hatem Mejjaoli for letting us know the papers \cite{BMMS,M}. The authors would like to thank the anonymous referees for their valuable suggestions and improvements.

\pdfbookmark[1]{References}{ref}
\LastPageEnding


\begin{thebibliography}{99}
\footnotesize\itemsep=0pt

\bibitem{B}
Ben~Sa\"{\i}d S., Strichartz estimates for {S}chr\"odinger--{L}aguerre
 operators, \href{https://doi.org/10.1007/s00233-014-9617-9}{\textit{Semigroup Forum}} \textbf{90} (2015), 251--269.

\bibitem{BKO}
Ben~Sa\"{\i}d S., Kobayashi T., {\O}rsted B., Laguerre semigroup and {D}unkl
 operators, \href{https://doi.org/10.1112/S0010437X11007445}{\textit{Compos. Math.}} \textbf{148} (2012), 1265--1336,
 \href{https://arxiv.org/abs/0907.3749}{arXiv:0907.3749}.

\bibitem{BMMS}
Boggarapu P., Mejjaoli H., Mondal S.S., Senapati P.J.K., Time-frequency
 analysis of {$(k, a)$}-generalized wavelet transform and applications,
 \href{https://doi.org/10.1063/5.0152806}{\textit{J.~Math. Phys.}} \textbf{64} (2023), 073504, 36~pages.

\bibitem{BT}
Bouclet J.-M., Tzvetkov N., Strichartz estimates for long range perturbations,
 \href{https://doi.org/10.1353/ajm.2007.0039}{\textit{Amer.~J. Math.}} \textbf{129} (2007), 1565--1609,
 \href{https://arxiv.org/abs/math.AP/0509489}{arXiv:math.AP/0509489}.

\bibitem{CK}
Christ M., Kiselev A., Maximal functions associated to filtrations,
 \href{https://doi.org/10.1006/jfan.2000.3687}{\textit{J.~Funct. Anal.}} \textbf{179} (2001), 409--425.

\bibitem{C}
Christianson H., Near sharp {S}trichartz estimates with loss in the presence of
 degenerate hyperbolic trapping, \href{https://doi.org/10.1007/s00220-013-1805-z}{\textit{Comm. Math. Phys.}} \textbf{324}
 (2013), 657--693, \href{https://arxiv.org/abs/1201.2464}{arXiv:1201.2464}.

\bibitem{DSS}
Donninger R., Schlag W., Soffer A., On pointwise decay of linear waves on a
 {S}chwarzschild black hole background, \href{https://doi.org/10.1007/s00220-011-1393-8}{\textit{Comm. Math. Phys.}}
 \textbf{309} (2012), 51--86, \href{https://arxiv.org/abs/0911.3179}{arXiv:0911.3179}.

\bibitem{FFFP1}
Fanelli L., Felli V., Fontelos M.A., Primo A., Time decay of scaling critical
 electromagnetic {S}chr\"odinger flows, \href{https://doi.org/10.1007/s00220-013-1830-y}{\textit{Comm. Math. Phys.}}
 \textbf{324} (2013), 1033--1067, \href{https://arxiv.org/abs/1203.1771}{arXiv:1203.1771}.

\bibitem{FFFP2}
Fanelli L., Felli V., Fontelos M.A., Primo A., Time decay of scaling invariant
 electromagnetic {S}chr\"odinger equations on the plane, \href{https://doi.org/10.1007/s00220-015-2291-2}{\textit{Comm. Math.
 Phys.}} \textbf{337} (2015), 1515--1533, \href{https://arxiv.org/abs/1405.1784}{arXiv:1405.1784}.

\bibitem{FZZ}
Fanelli L., Zhang J., Zheng J., Dispersive estimates for 2{D}-wave equations
 with critical potentials, \href{https://doi.org/10.1016/j.aim.2022.108333}{\textit{Adv. Math.}} \textbf{400} (2022), 108333,
 46~pages, \href{https://arxiv.org/abs/2003.10356}{arXiv:2003.10356}.

\bibitem{Fol}
Folland G.B., Harmonic analysis in phase space, \textit{Ann. of Math. Stud.},
 Vol. 122, \href{https://doi.org/10.1515/9781400882427}{Princeton University Press}, Princeton, NJ, 1989.

\bibitem{F}
Ford G.A., The fundamental solution and {S}trichartz estimates for the
 {S}chr\"odinger equation on flat euclidean cones, \href{https://doi.org/10.1007/s00220-010-1050-7}{\textit{Comm. Math. Phys.}}
 \textbf{299} (2010), 447--467, \href{https://arxiv.org/abs/0910.1795}{arXiv:0910.1795}.

\bibitem{GYZZ}
Gao X., Yin Z., Zhang J., Zheng J., Decay and {S}trichartz estimates in
 critical electromagnetic fields, \href{https://doi.org/10.1016/j.jfa.2021.109350}{\textit{J.~Funct. Anal.}} \textbf{282}
 (2022), 109350, 51~pages, \href{https://arxiv.org/abs/2003.03086}{arXiv:2003.03086}.

\bibitem{GV}
Ginibre J., Velo G., The global {C}auchy problem for the nonlinear
 {S}chr\"odinger equation revisited, \href{https://doi.org/10.1016/S0294-1449(16)30399-7}{\textit{Ann. Inst. H.~Poincar\'e Anal.
 Non Lin\'eaire}} \textbf{2} (1985), 309--327.

\bibitem{HZ}
Hassell A., Zhang J., Global-in-time {S}trichartz estimates on nontrapping,
 asymptotically conic manifolds, \href{https://doi.org/10.2140/apde.2016.9.151}{\textit{Anal. PDE}} \textbf{9} (2016),
 151--192, \href{https://arxiv.org/abs/1310.0909}{arXiv:1310.0909}.

\bibitem{H}
Howe R., The oscillator semigroup, in The {M}athematical {H}eritage of
 {H}ermann {W}eyl, \textit{Proc. Sympos. Pure Math.}, Vol.~48, \href{https://doi.org/10.1090/pspum/048/974332}{American
 Mathematical Society}, Providence, RI, 1988, 61--132.

\bibitem{IJ}
Ionescu A.D., Jerison D., On the absence of positive eigenvalues of
 {S}chr\"odinger operators with rough potentials, \href{https://doi.org/10.1007/s00039-003-0439-2}{\textit{Geom. Funct. Anal.}}
 \textbf{13} (2003), 1029--1081.

\bibitem{KT}
Keel M., Tao T., Endpoint {S}trichartz estimates, \textit{Amer.~J. Math.}
 \textbf{120} (1998), 955--980.

\bibitem{KM1}
Kobayashi T., Mano G., Integral formulas for the minimal representation of
 {${\rm O}(p,2)$}, \href{https://doi.org/10.1007/s10440-005-0464-2}{\textit{Acta Appl. Math.}} \textbf{86} (2005), 103--113.

\bibitem{KM}
Kobayashi T., Mano G., The inversion formula and holomorphic extension of the
 minimal representation of the conformal group, in Harmonic {A}nalysis,
 {G}roup {R}epresentations, {A}utomorphic {F}orms and {I}nvariant {T}heory,
 \textit{Lect. Notes Ser. Inst. Math. Sci. Natl. Univ. Singap.}, Vol.~12,
 \href{https://doi.org/10.1142/9789812770790_0006}{World Scientific Publishing}, Hackensack, NJ, 2007, 151--208,
 \href{https://arxiv.org/abs/math.RT/0607007}{arXiv:math.RT/0607007}.

\bibitem{KM2}
Kobayashi T., Mano G., The {S}chr\"odinger model for the minimal representation
 of the indefinite orthogonal group {${\rm O}(p,q)$}, \href{https://doi.org/10.1090/S0065-9266-2011-00592-7}{\textit{Mem. Amer. Math.
 Soc.}} \textbf{213} (2011), vi+132~pages, \href{https://arxiv.org/abs/0712.1769}{arXiv:0712.1769}.

\bibitem{KR}
Kumar V., Ruzhansky M., {$L^p$}-{$L^q$}-boundedness of {$(k,a)$}-{F}ourier
 multipliers with applications to nonlinear equations, \href{https://doi.org/10.1093/imrn/rnab256}{\textit{Int. Math. Res.
 Not.}} \textbf{2023} (2023), 1073--1093.

\bibitem{LLF}
Li S., Leng J., Fei M., Real {P}aley--{W}iener theorems for the
 {$(k,a)$}-generalized {F}ourier transform, \href{https://doi.org/10.1002/mma.6449}{\textit{Math. Methods Appl. Sci.}}
 \textbf{43} (2020), 6985--6994.

\bibitem{M}
Mejjaoli H., Dunkl--{S}chr\"odinger semigroups and applications, \href{https://doi.org/10.1080/00036811.2012.692780}{\textit{Appl.
 Anal.}} \textbf{92} (2013), 1597--1626.

\bibitem{MS}
Mondal S.S., Song M., Orthonormal Strichartz inequalities for the
 {$(k,a)$}-generalized {L}aguerre operator and {D}unkl operator,
 \href{https://arxiv.org/abs/2208.12015}{arXiv:2208.12015}.

\bibitem{Sh}
Sher D.A., Joint asymptotic expansions for {B}essel functions, \textit{Pure
 Appl. Anal.} \textbf{5} (2023), 461--505, \href{https://arxiv.org/abs/2203.06329}{arXiv:2203.06329},
 \href{https://doi.org/10.2140/paa.2023.5.461}.

\bibitem{So}
Sogge C.D., Fourier integrals in classical analysis, \textit{Cambridge Tracts
 in Math.}, Vol. 105, \href{https://doi.org/10.1017/CBO9780511530029}{Cambridge University Press}, Cambridge, 1993.

\bibitem{S}
Stein E.M., Harmonic analysis: real-variable methods, orthogonality, and
 oscillatory integrals, \textit{Princeton Math. Ser.}, Vol.~43, Princeton
 University Press, Princeton, NJ, 1993.

\bibitem{St}
Strichartz R.S., Restrictions of {F}ourier transforms to quadratic surfaces and
 decay of solutions of wave equations, \href{https://doi.org/10.1215/S0012-7094-77-04430-1}{\textit{Duke Math.~J.}} \textbf{44}
 (1977), 705--714.

\bibitem{TT}
Taira K., Tamori H., Integral representation for the $(k,a)$-generalized {L}aguerre semigroups, {i}n preparation.

\bibitem{T}
Tao T., Nonlinear dispersive equations: local and global analysis, \textit{CBMS
 Reg. Conf. Ser. Math.}, Vol.~106, \href{https://doi.org/10.1090/cbms/106}{American Mathematical Society}, Providence,
 RI, 2006.

\bibitem{Te}
Teng W., Hardy inequalities for fractional {$(k,a)$}-generalized harmonic
 oscillators, \textit{J.~Lie Theory} \textbf{32} (2022), 1007--1023,
 \href{https://arxiv.org/abs/2008.00804}{arXiv:2008.00804}.

\bibitem{W}
Watson G.N., A treatise on the theory of {B}essel functions, \textit{Cambridge Math. Lib.}, Cambridge University Press, Cambridge, 1995.

\bibitem{Y}
Yajima K., Existence of solutions for {S}chr\"odinger evolution equations,
 \href{https://doi.org/10.1007/BF01212420}{\textit{Comm. Math. Phys.}} \textbf{110} (1987), 415--426.

\bibitem{ZZ}
Zhang J., Zheng J., Strichartz estimates and wave equation in a conic singular
 space, \href{https://doi.org/10.1007/s00208-019-01892-7}{\textit{Math. Ann.}} \textbf{376} (2020), 525--581,
 \href{https://arxiv.org/abs/1804.02390}{arXiv:1804.02390}.

\end{thebibliography}
\end{document}